\documentclass[11pt]{amsart}
\usepackage[letterpaper,margin=1in]{geometry}
\usepackage[T1]{fontenc}
\usepackage{lmodern,amsmath,amssymb,mathtools,array}
\usepackage{microtype}
\usepackage{needspace}
\usepackage{xcolor}
\usepackage[colorlinks=true,linkcolor=blue!55!black,citecolor=blue!55!black,urlcolor=blue!55!black]{hyperref}
\hypersetup{pdftitle={A Ramanujan Pell K3 Surface and Cube Representations},pdfauthor={K. Srinivasa Raghava}}
\newtheorem{theorem}{Theorem}[section]
\newtheorem{proposition}[theorem]{Proposition}
\newtheorem{lemma}[theorem]{Lemma}
\newtheorem{corollary}[theorem]{Corollary}
\theoremstyle{remark}

\numberwithin{equation}{section}
\newcommand{\Q}{\mathbb Q}
\newcommand{\Z}{\mathbb Z}
\newcommand{\C}{\mathbb C}
\newcommand{\PP}{\mathbb P}
\newcommand{\Qbar}{\overline{\mathbb Q}}
\newcommand{\NS}{\operatorname{NS}}
\newcommand{\Hom}{\operatorname{Hom}}
\newcommand{\End}{\operatorname{End}}
\newcommand{\Triv}{\operatorname{Triv}}

\newcommand{\rank}{\operatorname{rank}}
\newcommand{\MW}{\operatorname{MW}}

\newcommand{\tors}{\mathrm{tors}}

\title[A Ramanujan Pell K3 surface]{A Ramanujan Pell K3 Surface\\ and Cube Representations}
\author{K. Srinivasa Raghava}
\address{Pie Mathematics Association, India}
\email{srinivasaraghavak@gmail.com}
\date{}
\begin{document}
\begin{abstract}
We classify all polynomial representations of
$K(u)=(1-13u+26u^2)^3+(6+182u^2)^3$
as a sum of two cubes over $\Qbar$, without an a priori degree
bound. The displayed representation is unique up to interchange
and independent multiplication by cube roots of unity. On the
elliptic K3 surface $y^2=x^3-432K(u)^2$, the resulting eighteen
sections are exactly those of minimum positive Shioda height
$4/3$. The proof combines an exact root-pairing calculation
with the degree--height relation on the genus-four cover
$w^3=K(u)$; it does not require the full Mordell--Weil rank.
For the rational representations arising from multiples of the
displayed section, we determine the reduced degrees and prove
strong divisibility of their common denominators, including
an exact formula for each primitive part. We also prove
nontorsion at every rational parameter and explain the Pell
origin of the sextic. The associated integral lattices and
symplectic quotient are treated in appendices. All proofs
are in characteristic zero.
\end{abstract}
\maketitle

\section{Introduction}
We study the polynomial representations of a sextic associated
with a Ramanujan identity and a Pell equation of discriminant
$65$. Set
\begin{equation}\label{intro:objects}
\begin{gathered}
p=1-13u+26u^2,\qquad q=6+182u^2,\\
F=p+q,\qquad G=p^2-pq+q^2,\qquad K=p^3+q^3=FG,\\
E_K:\quad y^2=x^3-432K^2.
\end{gathered}
\end{equation}
Let $X$ be its smooth minimal elliptic surface. Fix
$\omega=(-1+\sqrt{-3})/2$ and set
\begin{equation}\label{intro:sections}
\varrho(x,y)=(\omega x,y),\qquad
P=(12G,36(p-q)G),\qquad T=(0,12\sqrt{-3}K).
\end{equation}
We use the Shioda height pairing on sections. Our main result
classifies the polynomial representations and identifies them
with all sections of minimum positive height.

\begin{theorem}[Polynomial representations and minimum height]
\label{rep:polynomial}
The polynomial solutions of $U^3+V^3=K$ over $\Qbar$ are exactly
\[
 (U,V)=(\omega^i p,\omega^j q)
 \quad\hbox{or}\quad
 (U,V)=(\omega^i q,\omega^j p),\qquad 0\le i,j\le2.
\]
In particular, its only solutions in $\Q[u]^2$ are $(p,q)$ and
$(q,p)$. Under~\eqref{geom:fermat}, the eighteen pairs
correspond precisely to the sections of minimum positive height
$4/3$.
\end{theorem}

Reznick gives a general census of representations of complex
binary sextics as sums of two cubes of quadratic forms
\cite[Theorems~5.3--5.6]{Reznick}. His factorization criterion
\cite[Theorem~5.1]{Reznick} reduces their enumeration to products
of three pairwise nonproportional, linearly dependent quadratics.
The contribution here is the exact determination for $K$ and its
geometric interpretation. Theorem~\ref{rep:polynomial} imposes
no degree bound on the univariate solutions: a coefficient
argument first excludes cubic solutions with cancelling leading
terms, and a calculation in the splitting field leaves one
quadratic representation up to the stated symmetries. Thus the
eighteen ordered pairs constitute one equivalence class in
Reznick's census. The theorem further identifies them with every
section of minimum positive height, independently of the full
geometric Mordell--Weil rank.

The rational representations obtained from multiples of $P$
have equally explicit degree and divisor formulas.
For $m\ge1$, let $D_m$ be their monic reduced common denominator.
For positive integers $m,n$, Theorems~\ref{rep:rational}
and~\ref{rep:divisibility} give
\[
\deg D_m=2m^2-2,\qquad
\gcd(D_m,D_n)=D_{\gcd(m,n)},\qquad
\deg\Psi_m=2J_2(m)\quad(m>1),
\]
where $\Psi_m$ is the primitive part and
$J_2(m)=m^2\prod_{\ell\mid m}(1-\ell^{-2})$.
Reznick--Rouse establish the $\Z[\omega]$-module structure and
norm degree law for homogeneous rational solutions of
$f^3+g^3=x^3+y^3$ \cite{ReznickRouse}. Strong divisibility and
primitive-divisor results over function fields are established in
\cite{IngramMaheSilvermanStangeStreng,NaskreckiStreng}; for a
nonconstant point on a constant elliptic curve, primitive divisors
occur at every index in characteristic zero
\cite[Theorem~C]{NaskreckiStreng}. Here the pullback of the three
points at infinity on the Fermat cubic determines the common
denominator divisor exactly, including its multiplicities, after
the quadratic substitution $t=p/q$.

Theorem~\ref{spec:global} proves that $P_{u_0}$ is nontorsion
at every $u_0\in\PP^1(\Q)$. Over each fixed number field,
it has finitely many torsion specializations at smooth parameters,
and each cubic-twist class contains finitely many smooth parameters.
The Pell connection is treated in Section~\ref{arith:section},
following the identities of Hirschhorn and Harper
\cite{Hirschhorn,Harper}. With
$r_n+s_n(13+\sqrt{65})/2=(8+\sqrt{65})^n$, the associated
quadratic forms give
$a_n^3+b_n^3+c_n^3=561N_n^3+(-1)^{n+1}$.
Theorem~\ref{arith:constrained} proves completeness under the
constraints $a=2N-h$, $b+c=13N$, $h\in\{1,-1\}$,
by division in the quadratic integer ring.

Section~\ref{geom:section} supplies the surface and height
arguments needed for the main theorem, whose proof occupies
Section~\ref{rep:section}. Sections~\ref{rep:rationalsection}
and~\ref{spec:section} treat rational representations and
specialization. Appendices~\ref{lat:section} and~\ref{quot:section}
contain the Mordell--Weil rank analysis, integral lattices and
symplectic quotient. The full geometric rank remains undetermined
in $\{2,4,6\}$; all conclusions requiring rank two retain that
hypothesis.

\section{The elliptic surface and section heights}\label{geom:section}
Continue with~\eqref{intro:objects} and~\eqref{intro:sections}, and put
$R=\varrho(P)$. For $k\ne0$, the birational formulas
\begin{equation}\label{geom:fermat}
\begin{gathered}
x=\frac{12k}{U+V},\qquad y=\frac{36k(U-V)}{U+V},\\
U=\frac{36k+y}{6x},\qquad V=\frac{36k-y}{6x}
\end{gathered}
\end{equation}
extend to an isomorphism from the smooth projective model of
$U^3+V^3=k$ to
$E_k:y^2=x^3-432k^2$, taking $[1:-1:0]$ to the origin.

\begin{lemma}\label{geom:squarefree}
The sextic $K$ is squarefree. Its irreducible factorization over $\Q$ is
\[
F=208u^2-13u+7,\qquad
G=29068u^4+1690u^3+2067u^2+52u+31.
\]
The polynomials $p,q$ are coprime, including as binary quadratic forms.
\end{lemma}
\begin{proof}
The resultant of $p$ and $q$ is $2^3\cdot13^2\cdot137$, and their
leading coefficients are nonzero. Thus they have no common projective
zero. The discriminant of $F$ is $-5655$, so $F$ is irreducible.
Over $k=\Q(\omega)$,
\[
G=(p+\omega q)(p+\omega^2q).
\]
The discriminant of the first factor is $d=4433+3016\omega$, and
$d\bar d=3\cdot7^2\cdot13^2\cdot619$. If $d$ were a square in $k$,
its norm would be a rational square, which it is not. Both quadratic
factors are therefore irreducible over $k$. They are distinct and
exchanged by conjugation. A rational factor of $G$ would give a
conjugation-stable union of these two irreducible factors; hence $G$
is irreducible over $\Q$. A common root of any two of
$p+q,p+\omega q,p+\omega^2q$ would be a common root of $p,q$.
All three quadratics have nonzero discriminant. This proves the lemma.
\end{proof}

\begin{proposition}[Quadratic base change]\label{geom:basechange}
The map $t=p/q$ realizes $X$ as the quadratic base change of the
rational elliptic surface
\[
Y:\quad Y_1^2=X_1^3-432(t^3+1)^2,
\qquad x=q^2X_1,\quad y=q^3Y_1.
\]
The surface $X$ is K3 and has six singular geometric fibres, all of
type $\mathrm{IV}$. The deck involution is
\begin{equation}\label{geom:deck}
\begin{gathered}
\iota(u)=\frac{u+3}{91u-1},\qquad
\mu(u)=\frac{274}{(91u-1)^2},\\
(u,x,y)\longmapsto(\iota(u),\mu(u)^2x,\mu(u)^3y).
\end{gathered}
\end{equation}
It negates the holomorphic two-form and fixes exactly two smooth
elliptic fibres.
\end{proposition}
\begin{proof}
Coprimality makes $[p:q]$ a morphism of degree two. The discriminant
of $p-tq$, as a quadratic in $u$, is
\begin{equation}\label{geom:branch}
D(t)=65+1352t-4368t^2,
\end{equation}
with roots $(13\pm\sqrt{274})/84$. Squarefreeness of $K$ shows that
neither branch value is a root of $t^3+1$. The identity
$K=q^3(t^3+1)$ gives the claimed change of Weierstrass coordinates.

At each root of $t^3+1$, the equation of $Y$ is minimal with
$v(a_6)=2$ and $v(\Delta)=4$, giving a type-IV fibre. At infinity,
the substitution $(X_1,Y_1)=(s^{-2}\widetilde X,s^{-3}\widetilde Y)$,
$s=1/t$, gives a smooth fibre. Its fundamental line bundle is
$\mathcal O_{\PP^1}(1)$, so $Y$ is rational. The base change is
unramified at its three singular fibres, and doubles the fundamental
line bundle to $\mathcal O_{\PP^1}(2)$. Thus $X$ has six IV fibres.
The canonical bundle formula gives $K_X\cong\mathcal O_X$, and
$R^1f_*\mathcal O_X\cong\mathcal O_{\PP^1}(-2)$ gives
$h^1(X,\mathcal O_X)=0$. Hence $X$ is a K3 surface
\cite{Miranda,SchuttShiodaSurvey}.

Substitution gives $p\circ\iota=\mu p$, $q\circ\iota=\mu q$ and
$\iota'=-\mu$. Therefore the displayed map preserves the equation
and sends $du\wedge dx/y$ to its negative. It extends to the minimal
smooth surface by uniqueness of the relatively minimal elliptic
model. The fixed base points are $(1\pm\sqrt{274})/91$; at either
point $\mu=1$, so the whole fibre is fixed. Both fibres are smooth,
and no point over a different base point is fixed.
\end{proof}

Let $C$ be the smooth projective curve with function field
$\Q(u,w)$, $w^3=K(u)$, and let $E_0:Y_0^2=X_0^3-432$.

\begin{lemma}[Elliptic quotient]\label{geom:prym}
The curve $C$ has genus four. The morphism
\[
\pi:C\longrightarrow E_{\mathrm F},\qquad
(u,w)\longmapsto[p(u):q(u):w],\qquad
E_{\mathrm F}:U^3+V^3=W^3,
\]
has degree two.
\end{lemma}
\begin{proof}
The simple zeros of $K$ show that it is not a cube in
$\Qbar(u)$, so $w^3-K(u)$ is geometrically irreducible.
The map $C\to\PP^1_u$ has degree three, ramification index three
at each of the six zeros of $K$, and no ramification at infinity.
Riemann--Hurwitz gives $2g(C)-2=-6+12=6$. The formula for $\pi$
has no common zero; at infinity use $w/u^2,p/u^2,q/u^2$.
It is of degree two because $p/q$ is of degree two and the diagrams
with the degree-three maps to the two parameter lines commute.
Formula~\eqref{geom:fermat} with $k=1$ identifies
$E_{\mathrm F}$ with $E_0$.
\end{proof}

\begin{lemma}[Equivariant descent]\label{geom:descent}
Let $k$ be algebraically closed of characteristic zero, and let
$\tau(u,w)=(u,\omega w)$. There is an isomorphism
\begin{equation}\label{geom:hom}
\frac{E_K(k(u))}{E_K(k(u))_{\tors}}
\cong\{\phi\in\Hom_k(J(C),E_0):\phi\tau_*=\varrho\phi\}.
\end{equation}
The geometric torsion is $\langle T\rangle\cong\Z/3\Z$.
Over a characteristic-zero subfield $k_0\subset k$, its nonzero
elements are rational precisely when $\sqrt{-3}\in k_0$.
\end{lemma}
\begin{proof}
Over $k(C)$, write $x=w^2X_0$, $y=w^3Y_0$. A point over $k(u)$
then corresponds to a rational map $f:C\dashrightarrow E_0$ with
$f\tau=\varrho f$. Properness and smoothness extend it to a
morphism. Choose a point $b$ over a root of $K$. It is fixed by
$\tau$, so $f(b)\in\ker(1-\varrho)$. The morphism $f-f(b)$
factors uniquely through the Abel--Jacobi map $j_b:C\to J(C)$,
and $j_b\tau=\tau_*j_b$. This proves surjectivity onto the
displayed homomorphism group and identifies the kernel with the
constant equivariant maps. A torsion morphism is constant, since
its image is contained in a finite torsion scheme; conversely the
constant equivariant maps have values
$O,(0,\pm12\sqrt{-3})$. Twisting them back gives $O,\pm T$.
Their coordinates prove the field-of-definition statement.
\end{proof}

\begin{proposition}[Section heights]\label{lat:height}
The sections $P,R$ have height matrix
\[
\frac23\begin{pmatrix}2&-1\\-1&2\end{pmatrix}.
\]
Every nontorsion geometric section has height at least $4/3$,
and $P,R$ attain this minimum.
\end{proposition}
\begin{proof}
For a smooth projective curve $B$ over the algebraically closed
constant field, let $h_B(z)$ be the degree of the pole divisor of a rational
function $z$ on $B$. For a nontorsion section $Q$, untwisting gives
a nonconstant morphism $f_Q:C\to E_0$ by Lemma~\ref{geom:descent}.
Put $\nu:C\to\PP^1_u$, of degree three. The function $X_0$ on
$E_0$ has pole divisor $2O$, and multiplication by $2^j$ has degree
$4^j$. Consequently
\[
 h_C\bigl(X_0\circ[2^j]\circ f_Q\bigr)=2\cdot4^j\deg(f_Q).
\]
The untwisting identity gives
\[
 \nu^*x([2^j]Q)=w^2\bigl(X_0\circ[2^j]\circ f_Q\bigr).
\]
For nonzero rational functions $a,b$ on a curve,
$|h_B(ab)-h_B(a)|\le h_B(b)$, by the valuation definition of the
pole divisor. Thus multiplication by the fixed function $w^2$
changes these heights by a bounded amount independent of $j$.
Pullback by $\nu$ multiplies pole degrees by three. Using the
standard equality of the Shioda pairing with the function-field
canonical height in the doubled normalization
\cite{Shioda,SchuttShioda}, we obtain
\begin{equation}\label{geom:degreeheight}
 \langle Q,Q\rangle
 =\lim_{j\to\infty}4^{-j}h_{\PP^1_u}(x([2^j]Q))
 =\frac23\deg(f_Q)>0.
\end{equation}
The degree of $f_Q$ cannot be one: a degree-one morphism between
smooth projective curves is an isomorphism, whereas $C$ has genus
four and $E_0$ has genus one. Hence $\deg(f_Q)\ge2$, and every
nontorsion section has height at least $4/3$.

Under~\eqref{geom:fermat}, the section $P$ untwists to the
degree-two morphism $\pi$ of Lemma~\ref{geom:prym}. It is
therefore nontorsion, and~\eqref{geom:degreeheight} gives
$\langle P,P\rangle=4/3$.
The automorphism $\varrho$ preserves heights and satisfies
$1+\varrho+\varrho^2=0$ on the generic elliptic curve. Pairing this
relation with $P$ gives $\langle P,R\rangle=-2/3$.

\end{proof}

\section{Polynomial representations and minimum height}\label{rep:section}
We prove Theorem~\ref{rep:polynomial}. The first step classifies
all polynomial solutions; the degree--height identity then shows
that they account for every section of minimum positive height.

Let $E_{\mathrm F}$ denote the smooth projective Fermat cubic
\[
 E_{\mathrm F}:\quad \xi^3+\zeta^3=1,
\]
with origin $[1:-1:0]$.  Under its isomorphism with
$E_0:Y^2=X^3-432$, the three points at infinity form the subgroup
\[
 H=\ker(1-\varrho)=\{O,T_0,-T_0\},
 \qquad T_0=(0,12\sqrt{-3}).
\]
We also write $H$ for the reduced effective divisor consisting of these
three points. The functions $\xi$ and $\zeta$ both have pole divisor
$H$. Their zeros are three-torsion points. Indeed, on the projective
model $U^3+V^3=W^3$, the tangent at $[0:\varepsilon:1]$, where
$\varepsilon^3=1$, is $W=\varepsilon^2V$. Restriction of the cubic
equation to this line gives $U^3=0$, so the tangent meets the curve
triply at that point. The same argument applies to $V=0$.
At $O=[1:-1:0]$, the tangent $U+V=0$ gives $W^3=0$, so
the origin is also a flex.
The corresponding tangent-line divisors are linearly equivalent,
so $3Q\sim3O$ for each coordinate zero $Q$.

Write $\nu:C\to\PP^1_u$ for the degree-three cover and
$\pi:C\to E_{\mathrm F}$ for the degree-two morphism in
Lemma~\ref{geom:prym}.
Let $R_1,\ldots,R_6$ be the points over the six roots of $K$ and let
$I_1,I_2,I_3$ be the points over $u=\infty$.  Squarefreeness of $K$
and its nonzero leading coefficient give
\begin{equation}\label{rep:wdivisor}
 \operatorname{div}(w)=\sum_{i=1}^6R_i-2\sum_{j=1}^3I_j.
\end{equation}
At $R_i$ the functions $p,q$ are nonzero, and $w$ is a uniformizer.
The target coordinate $W/U$ at $\pi(R_i)\in H$ pulls back to
$w/p$, a uniformizer at $R_i$. Thus $\pi$ is unramified there.
Consequently
\begin{equation}\label{rep:branchpullback}
 \pi^*H=\sum_{i=1}^6R_i.
\end{equation}

\begin{proof}[Proof of Theorem~\ref{rep:polynomial}]
First we bound the degrees. Neither polynomial is zero, since
the squarefree sextic $K$ is not a cube. Put $d=\max(\deg U,\deg V)$.
If $d>2$, the leading terms of the two cubes must cancel, so
$\deg U=\deg V=d$. Multiplying $V$ by a cube root of unity,
which preserves the equation, arranges
$e=\deg(U+V)<d$. The leading coefficient of
$U^2-UV+V^2$ is then three times the square of that of $U$.
Thus $e+2d=6$, and necessarily $d=3$ and $e=0$.
Writing $U+V=c\ne0$ would give
\[
 K=3c(U-c/2)^2+c^3/4.
\]
We rule out this possibility by exact coefficient comparison.
The defining identity $K=p^3+q^3$ gives
\[
 K=6046144u^6-26364u^5+611442u^4-4225u^3
       +20241u^2-39u+217.
\]
If $K$ differed by a constant from a scalar multiple of a
square of a cubic, it would have the form
\[
 K-\kappa=6046144(u^3+au^2+bu+c_0)^2.
\]
The coefficients of $u^5,u^4,u^3$ successively force
\[
 a=-\frac3{1376},\qquad
 b=\frac{2489067}{49227776},\qquad
 c_0=-\frac{16199999}{67737419776}.
\]
The coefficient of $u^2$ gives a contradiction:
\[
 6046144(b^2+2ac_0)-20241
 =-\frac{24893400096903}{5210570752}\ne0.
\]
Since $\deg K=6$ also forces $d\ge2$, every polynomial solution
has $d=2$.

For the quadratic representations we use the factorization
criterion of Reznick \cite[Theorem~5.1]{Reznick}, expressed in
terms of pairings of the six roots. Work over
$k=\Q(\omega)$ and choose radicals $A,B,\Gamma$ with
\[
 A^2=4433+3016\omega,\qquad
 B^2=1417-3016\omega,\qquad \Gamma^2=-5655.
\]
The roots of the three factors $p+q,p+\omega q,p+\omega^2q$
are, respectively,
\begin{align*}
 f_1&=\frac{13+\Gamma}{416},& f_2&=\frac{13-\Gamma}{416},\\
 a_1&=\frac{13+A}{2(26+182\omega)},&
 a_2&=\frac{13-A}{2(26+182\omega)},\\
 b_1&=\frac{13+B}{2(26+182\omega^2)},&
 b_2&=\frac{13-B}{2(26+182\omega^2)}.
\end{align*}
These are six distinct roots since $K$ is squarefree.

The three radicands have independent squareclasses in $k$.
Indeed, the first two have the same rational norm
\[
 (4433+3016\omega)(1417-3016\omega)
 =91^2\cdot1857,\qquad1857=3\cdot619,
\]
which is not a rational square, so neither radicand is a square
in $k$. Their products with $-5655$ also have nonsquare rational
norm, since that norm is multiplied by $5655^2$.
For a rational number $r$, being a square in
$k=\Q(\sqrt{-3})$ is equivalent to $r$ or $-r/3$ being
a rational square: expand $(x+y\sqrt{-3})^2$ and use $xy=0$.
This excludes the remaining three products: $91^2\cdot1857$;
$-5655$, for which $5655/3=1885=5\cdot13\cdot29$;
and $-5655\cdot91^2\cdot1857$, for which division by $-3$
leaves an odd exponent of the prime $5$.
Thus $L=k(A,B,\Gamma)$ has degree eight over $k$, and
\begin{equation}\label{rep:radicalbasis}
 1,A,B,\Gamma,AB,A\Gamma,B\Gamma,AB\Gamma
\end{equation}
is a basis. The elementary squareclass criterion used here follows
by iterating the observation that, in a quadratic extension
$F(\sqrt{s})/F$, an element of $F$ becomes a square only if
its squareclass is $1$ or $s$: the mixed coefficient in
$(x+y\sqrt{s})^2$ must vanish.

Independent sign changes of $A,B,\Gamma$ act on the roots by the
transpositions $(a_1\ a_2)$, $(b_1\ b_2)$ and $(f_1\ f_2)$.
There is also an automorphism over $\Q$ sending
$\omega$ to $\omega^2$, interchanging $A,B$ and fixing $\Gamma$;
on the roots it acts as $(a_1\ b_1)(a_2\ b_2)$.
Their action on the fifteen partitions of the six roots into
three unordered pairs has four orbits. Representatives are
\begin{align*}
 \mathcal P_0&=((f_1,f_2),(a_1,a_2),(b_1,b_2)),\\
 \mathcal P_1&=((f_1,f_2),(a_1,b_1),(a_2,b_2)),\\
 \mathcal P_2&=((f_1,a_1),(f_2,a_2),(b_1,b_2)),\\
 \mathcal P_3&=((f_1,a_1),(f_2,b_1),(a_2,b_2)).
\end{align*}
If $f_1,f_2$ are paired together, the
remaining roots are paired either within their two original
pairs or across them, giving orbits of sizes one and two.
Otherwise their two partners belong either to the same original
pair or to different original pairs, giving four and eight
partitions, respectively. The sign changes and interchange
above act transitively in each case. This accounts for all
$1+2+4+8=15$ partitions.

For an ordered representative
$\mathcal P=((r_1,s_1),(r_2,s_2),(r_3,s_3))$, let
\[
 \delta(\mathcal P)=
 \det\begin{pmatrix}
 1&-r_1-s_1&r_1s_1\\
 1&-r_2-s_2&r_2s_2\\
 1&-r_3-s_3&r_3s_3
 \end{pmatrix}.
\]
It vanishes exactly when the three monic quadratics belonging
to these pairs are linearly dependent. Substitution of the
displayed roots, followed by the four relations
$\omega^2=-1-\omega$ and the displayed squares of $A,B,\Gamma$,
gives the following exact identities:
\begin{align*}
 930176\,\delta(\mathcal P_1)
   &=-388A\omega-341A+388B\omega+47B,\\
 930176\,\delta(\mathcal P_2)
   &=-388A\omega-341A+279\Gamma,\\
 24184576\,\delta(\mathcal P_3)
   &=3AB\Gamma-5044A\omega-4433A\\
   &\hspace{1.4em}{}-5044B\omega-611B-3627\Gamma.
\end{align*}
Each right-hand side is nonzero by the basis
\eqref{rep:radicalbasis}. Hence the only partition giving
linearly dependent quadratics is $\mathcal P_0$: its three
quadratics are scalar multiples of
$p+q,p+\omega q,p+\omega^2q$, which lie in the pencil spanned
by $p,q$.

Now factor a polynomial solution as
\[
 K=(U+V)(U+\omega V)(U+\omega^2V).
\]
All three factors have degree at most two, and their product
has degree six, so each has degree two. They are pairwise
coprime because $K$ is squarefree. Thus they determine a
partition of its roots into pairs. They lie in a vector space
of dimension at most two, so the partition must be
$\mathcal P_0$. Any two of these factors are linearly independent,
since they have disjoint roots. It follows that
$\langle U,V\rangle=\langle p,q\rangle$.
Write $U=ap+bq$ and $V=cp+dq$ with an invertible coefficient
matrix. Since $p/q$ is nonconstant, the identity $U^3+V^3=p^3+q^3$
implies the homogeneous polynomial identity
\[
 (aX+bY)^3+(cX+dY)^3=X^3+Y^3.
\]
There are at most eighteen invertible linear transformations
preserving $X^3+Y^3$: their projective transformations permute
its three distinct zeros in $\PP^1$, giving at most six
possibilities, and each projective transformation has at most
three scalar lifts preserving the cubic. Independently multiplying
$X,Y$ by cube roots of unity and optionally interchanging them
already gives eighteen such transformations. These are therefore
all of them and yield precisely the pairs in the statement.
They are distinct because $p,q$ are nonzero and nonproportional.
Since a coefficient multiplier $\omega^i$ is rational only
when $i=0$, exactly the two stated pairs are rational.

For the height assertion, let $Q$ be any section of height $4/3$.
Equation~\eqref{geom:degreeheight} gives $\deg(f_Q)=2$ for its
untwisted morphism $f_Q:C\to E_{\mathrm F}$. Each of the six
points $R_i$ above $K=0$ is fixed by $\tau$, so equivariance
implies $f_Q(R_i)\in H$. The effective divisor $f_Q^*H$ has
degree six and contains these six distinct points. Consequently
\[
 f_Q^*H=\sum_{i=1}^6R_i.
\]
Multiplication of the Fermat coordinates by $w$ therefore cancels
every finite pole, by~\eqref{rep:wdivisor}. The resulting functions
descend to $\Qbar(u)$ by equivariance and are polynomials. The
classification just proved puts $Q$ among the eighteen sections
in the statement. Conversely, each displayed polynomial pair
induces a degree-two morphism on $C$: it is $\pi$ followed by a
Fermat automorphism. Equation~\eqref{geom:degreeheight} gives its
height $4/3$, which is the minimum by Proposition~\ref{lat:height}.
\end{proof}

\section{Rational representations and divisibility}\label{rep:rationalsection}
The polynomial classification is complete for the full geometric
section group. The formulas in this section concern the subgroup
$\mathcal M=\Z[\omega]P\oplus\langle T\rangle$; no assertion that
it is the full group is needed. We first record a rational torsion
criterion, which will control the poles and zeros at infinity.

\begin{lemma}[Rational torsion and a nontorsion criterion]\label{spec:torsion}
For $k\in\Q^\times$, let $E_k:y^2=x^3-432k^2$. Then
\[
 E_k(\Q)_{\tors}\simeq
 \begin{cases}
 \Z/3\Z,&k\in\Q^{\times3},\\
 \Z/2\Z,&k\in2\Q^{\times3},\\
 0,&\text{otherwise}.
 \end{cases}
\]
If $a,b\in\Q^\times$, $a\ne b$, and $a^3+b^3\ne0$, then
\[
 \bigl(12(a^2-ab+b^2),\,36(a-b)(a^2-ab+b^2)\bigr)
 \in E_{a^3+b^3}(\Q)
\]
has infinite order.
\end{lemma}

\begin{proof}
Olson's classification of rational torsion on CM elliptic curves
\cite{Olson}, applied to this curve with $j=0$, gives the possibilities
\[
0,\quad \Z/2,\quad \Z/3,\quad \Z/4,\quad
\Z/6,\quad (\Z/2)^2.
\]
The last possibility is excluded because $x^3-432k^2$
has only one real root, hence at most one rational root.
To exclude order four, write a rational point of order two as
$(r,0)$, so the equation is $y^2=x^3-r^3$ with $r\ne0$.
If $2Q=(r,0)$ for a rational point $Q=(x,y)$, then $y\ne0$ and
the duplication formula gives
\[
\frac{x^4+8r^3x}{4(x^3-r^3)}=r,
\qquad
(x^2-2rx-2r^2)^2=0.
\]
Consequently $(x-r)^2=3r^2$, impossible over $\Q$.
The torsion group is therefore cyclic of order dividing six.

A rational point of order two requires $x^3=432k^2$.
Taking prime valuations shows that $v_2(k)\equiv1\pmod3$ and
$v_\ell(k)\equiv0\pmod3$ for every prime $\ell\ne2$.
This is equivalent to $k=2v^3$ for $v\in\Q^\times$, in which
case $(12v^2,0)$ has order two. The three-division polynomial is
\[
 \psi_3(x)=3x(x^3-1728k^2).
\]
The root $x=0$ has no rational $y$-coordinate. The other factor has a
rational root exactly when every prime valuation of $k$ is a
multiple of three, equivalently $k=v^3$ for some $v\in\Q^\times$.
Then $(12v^2,\pm36v^3)$ are points of order three. Since a
nonzero rational number cannot be both a cube and twice a cube,
the two torsion cases are disjoint and the classification follows.

For the last assertion put $B=a^2-ab+b^2>0$. The displayed point
is finite and has nonzero $y$-coordinate, so it is neither the identity
nor a point of order two. Its $x$-coordinate is nonzero, and
\[
 (12B)^3-1728(a^3+b^3)^2=-5184abB^2\ne0.
\]
It is not a point of order three either. The classification proves
that it has infinite order.
\end{proof}

For $\alpha=a+b\omega\in\Z[\omega]$, put
$N(\alpha)=a^2-ab+b^2$.  The notation $\alpha P$ means
$aP+b\varrho(P)$. For $\alpha\ne0$, the complex-multiplication
endomorphism $[\alpha]$ has degree $N(\alpha)$: under a complex
uniformization $E_{\mathrm F}(\C)\cong\C/\Lambda$, its kernel has
order $[\Lambda:\alpha\Lambda]=\alpha\bar\alpha=N(\alpha)$.
We consider the representations arising from
$\mathcal M=\Z[\omega]P\oplus\langle T\rangle$.

\begin{theorem}[Representations by rational functions]
\label{rep:rational}
For $m\in\Z\setminus\{0\}$ write $(x_m,y_m)=mP$. The formulas
\begin{equation}\label{rep:rationalformulas}
 U_m=\frac{36K+y_m}{6x_m},\qquad
 V_m=\frac{36K-y_m}{6x_m}
\end{equation}
give distinct ordered solutions of $U^3+V^3=K$ in $\Q(u)$.
There are polynomials $A_m,B_m,D_m\in\Q[u]$, with $D_m$ monic and
$\gcd(A_m,D_m)=\gcd(B_m,D_m)=1$, such that
\begin{equation}\label{rep:degrees}
 U_m=\frac{A_m}{D_m},\qquad V_m=\frac{B_m}{D_m},\qquad
 A_m^3+B_m^3=KD_m^3,
\end{equation}
and
\begin{equation}\label{rep:exactdegrees}
 \deg A_m=\deg B_m=2m^2,\qquad \deg D_m=2m^2-2.
\end{equation}
More generally, each section $\alpha P+jT$, with
$0\ne\alpha\in\Z[\omega]$ and $j\in\{0,1,2\}$,
gives an ordered solution in $\Qbar(u)$ by the same inverse formulas.
Different pairs $(\alpha,j)$ give different ordered solutions.
Its reduced common denominator has degree $2N(\alpha)-2$, and each
reduced numerator has degree $2N(\alpha)$.
\end{theorem}

\begin{proof}
For an affine representation, $U+V\ne0$, since otherwise $K=0$ as a
rational function.  The identities
\[
 x=\frac{12K}{U+V},\qquad
 y=\frac{36K(U-V)}{U+V}
\]
give a point on $E_K$ with $x\ne0$.  Conversely, every point of
$E_K$ other than $O,T,-T$ has $x\ne0$, and the inverse formulas in
\eqref{rep:rationalformulas} give an affine representation.  These
constructions are inverse to one another. Use the maps
$\nu:C\to\PP^1_u$ and $\pi:C\to E_{\mathrm F}$ of
Section~\ref{rep:section}, of degrees three and two. Untwisting the
section $\alpha P+jT$ over $\Qbar(C)$ gives the morphism
\[
h=[\alpha]\circ\pi+jT_0:C\longrightarrow E_{\mathrm F}
\]
of degree $2N(\alpha)>0$. Hence this section is never one of
$O,T,-T$, which untwist to constant maps, so its inverse formulas
are defined. If two parameter pairs $(\alpha,j),(\beta,k)$ gave
the same section, then
\[
[\alpha-\beta]\circ\pi+(j-k)T_0=0.
\]
For $\alpha\ne\beta$ the left side is a nonconstant map, of
degree $2N(\alpha-\beta)$, a contradiction. Thus $\alpha=\beta$;
the exact order three of $T_0$ then gives $j=k$.
The rationality assertion for $mP$ follows from $P\in E_K(\Q(u))$.

To compute the reduced degrees, pull the represented functions
back to $C$:
\begin{equation}\label{rep:untwisted}
 U=w(\xi\circ h),\qquad V=w(\zeta\circ h).
\end{equation}
Thus $\xi\circ h$ and $\zeta\circ h$ have the same pole divisor
$h^*H$, of degree $6N(\alpha)$.

Every endomorphism $[\alpha]$, with $\alpha\in\Z[\omega]$, maps
$H$ into itself, since it commutes with $\varrho$.
Translation by $jT_0$ permutes $H$, and $[\alpha]$ is \'etale
for $\alpha\ne0$ in characteristic zero.
By~\eqref{rep:branchpullback}, $h^*H$ therefore has coefficient
one at each $R_i$. Multiplication
by $w$ cancels each of these poles exactly.

We next check the points at infinity.  Put
$p_2=\operatorname{lc}(p)$ and $q_2=\operatorname{lc}(q)$.
These are distinct nonzero positive rational numbers.  The point
associated with $p_2,q_2$ on
$E_{p_2^3+q_2^3}$ is nontorsion by Lemma~\ref{spec:torsion}.
After an untwisting isomorphism its image is $\pi(I_j)$, for each
of the three choices of cube root of $p_2^3+q_2^3$.
Hence every $\pi(I_j)$ is nontorsion.  A nonzero isogeny followed
by a torsion translation preserves nontorsion: if the image were
torsion, a multiple of the original point would lie in the finite
kernel of the isogeny.  Thus $h(I_j)$ is nontorsion as well.  It
belongs neither to $H$ nor to the zero divisor of either Fermat
coordinate.  Equations~\eqref{rep:untwisted} and
\eqref{rep:wdivisor} show that both $U$ and $V$ have a pole of
order exactly two at each $I_j$.

The multiplication in \eqref{rep:untwisted} has therefore removed
six simple poles and introduced three double poles.  The two
functions have identical pole divisors of total degree
$6N(\alpha)$ on $C$.  Since they are pulled back through the
degree-three map $\nu$, their pole divisors on $\PP^1_u$ have
degree $2N(\alpha)$.  The pole at infinity has order two.  Their
finite pole divisors are equal, so the monic reduced denominators
are equal, of degree $2N(\alpha)-2$.  Each reduced numerator has
degree two more than this.  The rational case follows by taking
$\alpha=m$.
\end{proof}

For $m\ge1$, let $D_m$ be the monic denominator of
Theorem~\ref{rep:rational}, and put
$\Delta_m=\operatorname{div}_0(D_m)$ on $\PP^1_u$.
In particular, $D_1=1$ and $\Delta_1=0$.

\begin{theorem}[Strong divisibility and primitive divisors]
\label{rep:divisibility}
For all positive integers $m,n$,
\begin{equation}\label{rep:gcd}
 \gcd(D_m,D_n)=D_{\gcd(m,n)},
\end{equation}
where the gcd is monic.  For $m>1$, define initially as a rational
function
\[
 \Psi_m=\prod_{d\mid m}D_d^{\mu(m/d)}.
\]
Then $\Psi_m\in\Q[u]$ is monic, the polynomials $\Psi_m$ for
$m>1$ are pairwise coprime, and
\begin{equation}\label{rep:primitiveproduct}
 D_m=\prod_{\substack{d\mid m\\d>1}}\Psi_d.
\end{equation}
Moreover,
\begin{equation}\label{rep:jordan}
 \deg\Psi_m=2J_2(m)
 =2m^2\prod_{\ell\mid m}\left(1-\ell^{-2}\right),\qquad m>1,
\end{equation}
where $\mu$ is the M\"obius function and the last product is over
the rational primes dividing $m$.  In particular, every $D_m$ with
$m>1$ has an irreducible factor in $\Q[u]$ that divides none of
$D_1,\ldots,D_{m-1}$.
\end{theorem}

\begin{proof}
Work first over $\Qbar$.  Translation by an element of $H$ preserves
the divisor $H$, so the pole calculation above gives the exact
effective-divisor identity
\begin{equation}\label{rep:denominatordivisor}
 \nu^*\Delta_m=\pi^*\bigl([m]^*H-H\bigr).
\end{equation}
Indeed, the pole divisor of $U_m$ on $C$ is
\[
 \pi^*[m]^*H-\pi^*H+2\sum_{j=1}^3I_j.
\]
Here the first subtraction is the cancellation at the $R_i$, by
\eqref{rep:branchpullback}; none of the $I_j$ occurs in either
$\pi^*[m]^*H$ or $\pi^*H$, by the nontorsion argument above.
On the other hand, the pullback of the pole divisor of
$A_m/D_m$ is
$\nu^*\Delta_m+2\sum_jI_j$, because $\nu$ is unramified at
infinity.  Subtracting the common infinity divisor proves
\eqref{rep:denominatordivisor}.

Let $z\in C(\Qbar)$ lie outside $\pi^{-1}(H)$, and let
$\overline{\pi(z)}$ be the image of $\pi(z)$ in the quotient
elliptic curve $E_{\mathrm F}/H$.  If this image is nontorsion,
$z$ has coefficient zero on the right of
\eqref{rep:denominatordivisor} for every $m$.  If it has finite
order $r_z$, then $r_z>1$, and $z$ has coefficient
\begin{equation}\label{rep:multiplicity}
 \begin{cases}
 e_\pi(z),& r_z\mid m,\\
 0,& r_z\nmid m.
 \end{cases}
\end{equation}
Indeed, $[m]\pi(z)\in H$ is equivalent to $r_z\mid m$.
The divisor $[m]^*H$ is reduced since $[m]$ is \'etale, so its
pullback multiplicity is exactly the ramification index
$e_\pi(z)$, independent of $m$.  Points of $\pi^{-1}(H)$ have
coefficient zero in the difference for every $m$, and the points
at infinity have already been excluded.

Taking the minimum of the multiplicities for $m,n$ in
\eqref{rep:multiplicity} gives the multiplicity for
$\gcd(m,n)$.  Pullback by a finite surjective morphism is
injective on divisors and commutes with their coefficientwise
minimum: at a point above a fixed base point it multiplies all
coefficients by the same positive ramification index.  Thus
\eqref{rep:denominatordivisor} gives
$\gcd(\Delta_m,\Delta_n)=\Delta_{\gcd(m,n)}$, proving
\eqref{rep:gcd} after monic normalization.

A finite root occurring in some $D_m$ has a unique first index
$r>1$ and occurs exactly
at indices divisible by $r$, with a fixed positive multiplicity.
Indeed, away from the $R_i$, the three points above a finite base
point form a single orbit under $\tau(u,w)=(u,\omega w)$.
The equality $\pi\tau=\varrho\pi$ and preservation of $H$ by
$\varrho$ show that their images in $E_{\mathrm F}/H$ have the same
order. Equation~\eqref{rep:multiplicity}, together with the
unramifiedness of $\nu$ there, gives the asserted occurrence pattern
and constant multiplicity.
For such a root of multiplicity $e$, its exponent in $\Psi_m$ is
\[
 e\sum_{\substack{d\mid m\\r\mid d}}\mu(m/d)
 =\begin{cases}e,&m=r,\\0,&m\ne r.\end{cases}
\]
The equality follows from $\sum_{d\mid n}\mu(d)=0$ for $n>1$
and $1$ for $n=1$.  Thus $\Psi_m$ has no finite poles, and, as
it belongs to $\Q(u)$, it is a polynomial in $\Q[u]$.
Its leading coefficient is one because every $D_d$ is monic.
The same exponent calculation proves pairwise coprimality and
\eqref{rep:primitiveproduct}, with multiplicities.

Taking degrees and applying \eqref{rep:exactdegrees}, for $m>1$
we obtain
\[
 \deg\Psi_m
 =\sum_{d\mid m}\mu(m/d)(2d^2-2)
 =2\sum_{d\mid m}\mu(m/d)d^2
 =2m^2\prod_{\ell\mid m}(1-\ell^{-2}).
\]
This is positive.  Any irreducible factor of $\Psi_m$ is
therefore a factor of $D_m$, and its first index of occurrence
is $m$, proving the final assertion.
\end{proof}

\section{Specialization and cubic twists}\label{spec:section}

For $[r:s]\in\PP^1$, write
\begin{align*}
 p_h(r,s)&=r^2-13rs+26s^2,&
 q_h(r,s)&=6r^2+182s^2,\\
 K_h(r,s)&=p_h(r,s)^3+q_h(r,s)^3,&
 G_h(r,s)&=p_h(r,s)^2-p_h(r,s)q_h(r,s)+q_h(r,s)^2.
\end{align*}
When $K_h(r,s)\ne0$, define
\[
 \mathcal E_{[r:s]}:\ y^2=x^3-432K_h(r,s)^2,
 \qquad
 P_{[r:s]}=\bigl(12G_h(r,s),\,36(p_h(r,s)-q_h(r,s))G_h(r,s)\bigr).
\]
For $c\ne0$, replacing $(r,s)$ by $(cr,cs)$ multiplies $K_h$ by $c^6$
and gives the isomorphism $(x,y)\mapsto(c^4x,c^6y)$. Thus
$\mathcal E_{[r:s]}$, $P_{[r:s]}$, and the cube class of $K_h$ depend only
on $[r:s]$, up to this isomorphism. We also write $\mathcal E_{u_0},P_{u_0}$,
where $u_0=s/r$ for $r\ne0$ and $u_0=\infty$ for $[r:s]=[0:1]$.
The latter is the smooth minimal fibre at infinity.

Let $h$ denote the absolute logarithmic Weil height on $\PP^1$.
For an elliptic curve over a number field or over $\Qbar$, put
\begin{equation}\label{spec:normalization}
 H_E(Q)=\lim_{j\to\infty}4^{-j}h_x([2^j]Q),\qquad
 h_x(Q)=h(x(Q)),\qquad h_x(O)=0.
\end{equation}
Here the $x$-coordinate map is extended by $x(O)=\infty$ and
$h(\infty)=0$. This is twice the canonical height in
\cite[Section~VIII.9]{Silverman}.
Its function-field counterpart is the Shioda height pairing.

\begin{lemma}[Uniform height comparison]\label{spec:comparison}
For every finite algebraic $u$ with $K(u)\ne0$,
\begin{equation}\label{spec:uniform}
 H_{E_{K(u)}}(P(u))
 =\frac13h\!\left(\frac{1728G(u)}{F(u)^2}\right)+O(1)
 =\frac43h(u)+O(1),
\end{equation}
where both bounded terms are independent of $u$ and of its field
of definition.
\end{lemma}

\begin{proof}
Choose $w\in\Qbar$ with $w^3=K(u)$. The isomorphism
\[
 E_{K(u)}\longrightarrow E_0:\quad
 (x,y)\longmapsto(x/w^2,y/w^3),\qquad
 E_0:Y^2=X^3-432,
\]
sends $P(u)$ to a point $R_u$ satisfying
\begin{equation}\label{spec:fixedx}
 X(R_u)^3=\frac{1728G(u)^3}{K(u)^2}
 =\frac{1728G(u)}{F(u)^2}.
\end{equation}
Duplication on $E_0$ induces the degree-four morphism
\[
 \delta(X)=\frac{X^4+3456X}{4X^3-1728}
\]
of $\PP^1$. Its homogeneous numerator and denominator have no
common zero: at infinity the numerator is nonzero, while a finite
common zero would have $X^3=432$ and then
$X^4+3456X=3888X\ne0$.
The height estimate for a fixed morphism of projective space gives
a constant $c_0$ such that
\[
 |h(\delta(t))-4h(t)|\le c_0
 \qquad(t\in\PP^1(\Qbar)).
\]
The constant is independent of the field of definition. This
uniformity follows from upper bounds for the fixed homogeneous
forms and lower bounds from B\'ezout identities, combined with
the product formula; see \cite[Theorem~VIII.5.6]{Silverman}.
Telescoping along the duplication orbit gives
\[
 |H_{E_0}(R)-h(X(R))|
 \le\sum_{j\ge0}\frac{c_0}{4^{j+1}}=\frac{c_0}{3}
\]
for every algebraic point $R$, including points whose duplication
orbit reaches the identity.

Canonical height is invariant under the displayed isomorphism:
for the fixed parameter $u$, scaling an $x$-coordinate by $w^2$ changes
its height by at most $2h(w)$, and division by $4^j$ makes this
error tend to zero. Since $h(z^3)=3h(z)$, equation
\eqref{spec:fixedx} proves the first equality in
\eqref{spec:uniform}.

The rational map $1728G/F^2$ has degree four. Indeed, the
homogeneous quadratic forms $p_h,q_h$ have no common zero; at a
zero of $p_h+q_h$ one has $p_h=-q_h\ne0$, so
$p_h^2-p_hq_h+q_h^2=3q_h^2\ne0$. Thus its homogeneous quartic
numerator and denominator have no common zero. Applying the
fixed-morphism height estimate once more gives
$h(1728G(u)/F(u)^2)=4h(u)+O(1)$ uniformly, proving the second
equality.
\end{proof}

\Needspace{8\baselineskip}
\begin{theorem}[Specialization and cubic twists]\label{spec:global}
The following assertions hold.
\begin{enumerate}
\item Every fibre above $u_0\in\PP^1(\Q)$ is smooth, and $P_{u_0}$ has
infinite order. The specialization map
\[
 \Z P\longrightarrow\mathcal E_{u_0}(\Q)
\]
is injective at every such parameter, with infinity interpreted
in the minimal homogeneous model.
\item For every number field $L$, the map
\[
 \{[r:s]\in\PP^1(L):K_h(r,s)\ne0\}
 \longrightarrow L^\times/L^{\times3},\qquad
 [r:s]\longmapsto[K_h(r,s)]
\]
has finite fibres.
\item For every fixed number field $L$, only finitely many smooth
$u_0\in\PP^1(L)$ have $P_{u_0}$ torsion. Consequently every infinite
set of smooth $L$-rational parameters yields infinitely many
cubic-twist classes whose curves have $L$-rank at least one.
\end{enumerate}
\end{theorem}

\begin{proof}
For a nonzero rational pair $(r,s)$, one has $q_h(r,s)>0$.
The quadratic form $p_h$ has discriminant $65$, a rational
nonsquare, so $p_h(r,s)\ne0$ for such a pair. Also
\[
 p_h-q_h=-(5r^2+13rs+156s^2)<0,
 \qquad
 p_h+q_h=7r^2-13rs+208s^2>0:
\]
the quadratic forms on the right have positive leading
coefficients and discriminants $-2951$ and $-5655$, respectively.
The identity
\[
 K_h=(p_h+q_h)(p_h^2-p_hq_h+q_h^2)
\]
then gives $K_h>0$. Hence every rational fibre is smooth.
Lemma~\ref{spec:torsion}, with $a=p_h$ and $b=q_h$, proves that
$P_{u_0}$ has infinite order. This includes $[r:s]=[0:1]$ and hence
the fibre at infinity. Since specialization is a group
homomorphism at a smooth fibre and $P_{u_0}$ has infinite order, its
restriction to $\Z P$ is injective.

For the second assertion, rescaling $(r,s)$ changes $K_h$ by a
sixth power, so the cube class is well defined. Fix
$c\in L^\times$. Each finite parameter in the fibre of $[c]$
gives an $L$-rational point on
\[
 C_c:\quad cw^3=K(u).
\]
The polynomial $K$ is squarefree of degree six by
Lemma~\ref{geom:squarefree}. Each of its zeros has order one, so
$K/c$ is not a cube in $\overline L(u)$. Hence the cubic
$w^3-K/c$ is irreducible over $\overline L(u)$, and the smooth
projective model of $C_c$ is geometrically integral. Its map to
$\PP^1_u$ has degree three and is totally ramified at the six
finite zeros of $K$. It is unramified at every other finite point,
since there $w\ne0$ and the derivative $3cw^2$ is nonzero.
At infinity put $z=1/u$ and $v=wz^2$. The equation becomes
$cv^3=z^6K(1/z)$, whose value at $z=0$ is the nonzero leading
coefficient of $K$. Over $\overline L$ it has three distinct
nonzero solutions for $v$, with nonzero derivative $3cv^2$.
Thus there are three unramified geometric points over infinity.
Riemann--Hurwitz gives
\[
 2g(C_c)-2=3(-2)+6(3-1)=6,
\]
so $g(C_c)=4$. Faltings's theorem gives finitely many
$L$-rational points on $C_c$ \cite{Faltings}. Their images account
for all the finite parameters in the cube-class fibre, and at
most one further parameter is infinity. This proves finiteness.

For the third assertion, a torsion specialization has canonical
height zero. Lemma~\ref{spec:comparison} consequently bounds
$h(u_0)$ by an absolute constant at all finite algebraic parameters
where this occurs. Northcott's theorem
\cite[Theorem~VIII.5.11]{Silverman} makes the set of such
parameters finite over each fixed number field $L$; infinity adds
at most one point. An infinite set of smooth $L$-rational
parameters therefore contains infinitely many nontorsion
specializations. The second assertion shows that these
occupy infinitely many cube classes.
Finally, for $k,k'\in L^\times$, the elliptic curves $E_k$ and $E_{k'}$
are isomorphic over $L$ exactly when $k/k'$ is a cube in $L$.
Indeed, an origin-preserving isomorphism between short Weierstrass
equations in characteristic zero is a scaling of the $x$-coordinate
and $y$-coordinate by a square and a cube of the same element of
$L^\times$. It therefore requires $(k/k')^2$ to be a sixth power;
then $k/k'=\pm v^3$, and the sign is itself a cube. This
identifies the asserted cube classes with distinct $L$-isomorphism
classes in this cubic-twist family. The conclusion also holds for
isomorphisms of the underlying curves: translation by the negative
of the image of $O$ turns any such $L$-isomorphism into one that
preserves the origin.
\end{proof}

\begin{corollary}[Rational specializations]\label{rep:norationalpoles}
For every $m\ge1$ and every $u_0\in\Q$, one has $D_m(u_0)\ne0$.
\end{corollary}

\begin{proof}
By Theorem~\ref{spec:global}, the fibre at $u_0$ is smooth and
$P_{u_0}$ is nontorsion. A pole of the Fermat coordinates at $u_0$
would specialize $mP$ to one of the three points at infinity
of the smooth Fermat model.  These points form a torsion
subgroup, which would imply that $P_{u_0}$ is torsion. This is
impossible. Since $D_m$ is reduced, it cannot vanish at $u_0$.
\end{proof}

\section{The Pell equation}\label{arith:section}

Put
\[
 d=\sqrt{65},\qquad \theta=\frac{13+d}{2},\qquad
 \eta=8+d=2\theta-5,\qquad \alpha=\eta^2.
\]
The ring $\mathcal O=\Z[\theta]$ is the full ring of integers of
$\Q(d)$. Positivity of a unit refers to the real embedding
$d=+\sqrt{65}$. Since $\theta^2-13\theta+26=0$ and
$\operatorname{N}_{\Q(d)/\Q}(\eta)=-1$, both $\eta$ and $\eta^{-1}$
belong to $\mathcal O$. Thus the equations
\begin{equation}\label{arith:orbit}
 r_n+s_n\theta=\eta^n\qquad(n\in\Z)
\end{equation}
define integers $r_n,s_n$. Set
\begin{align}
 N_n&=r_n^2+26s_n^2,&
 a_n&=r_n^2-13r_ns_n+26s_n^2,\notag\\
 b_n&=6r_n^2+182s_n^2,&
 c_n&=7r_n^2+156s_n^2.\label{arith:bases}
\end{align}

\begin{proposition}[An identity in three cubes]\label{arith:construction}
For every $n\in\Z$, the integers $a_n,b_n,c_n,N_n$ are positive and
\begin{equation}\label{arith:identity65}
 a_n^3+b_n^3+c_n^3=561N_n^3-(-1)^n.
\end{equation}
They satisfy
\begin{equation}\label{arith:linear}
 a_n=2N_n-(-1)^n,\qquad b_n+c_n=13N_n.
\end{equation}
\end{proposition}

\begin{proof}
Taking the norm in \eqref{arith:orbit} gives
\begin{equation}\label{arith:norm}
 r_n^2+13r_ns_n+26s_n^2=(-1)^n.
\end{equation}
For $N=X+Y$, expansion of the two symmetric pairs of cubes gives the
polynomial identity
\begin{align*}
 &(N+Z)^3+(N-Z)^3+(6X+7Y)^3+(7X+6Y)^3\\
 &\hspace{25mm}=561N^3+3N(2Z^2-13XY).
\end{align*}
Apply it with $X=r_n^2$, $Y=26s_n^2$, and $Z=13r_ns_n$.
Then $2Z^2=13XY$, $N+Z=(-1)^n$, and $N-Z=a_n$, proving
\eqref{arith:identity65}. The identities in \eqref{arith:linear} follow
from \eqref{arith:norm} and \eqref{arith:bases}. The pair $(r_n,s_n)$
is nonzero, so $N_n\ge1$ and $a_n=2N_n-(-1)^n>0$; positivity of
$b_n$ and $c_n$ follows from their positive-definite expressions.
\end{proof}

Multiplication by $\eta=2\theta-5$ gives
\begin{equation}\label{arith:matrix}
 \binom{r_{n+1}}{s_{n+1}}
 =\begin{pmatrix}-5&-52\\2&21\end{pmatrix}
   \binom{r_n}{s_n},\qquad (r_0,s_0)=(1,0).
\end{equation}
The determinant is $-1$, and both the matrix and its integral inverse
preserve $(r,s)\equiv(1,0)\pmod2$. In particular, every $r_n$ is odd
and every $s_n$ is even.

\begin{proposition}[The unit group]\label{arith:unitcomplete}
Every integral solution of $r^2+13rs+26s^2=\pm1$ is
$\pm(r_n,s_n)$ for a unique $n\in\Z$. Every positive unit of
$\mathcal O$ is a power of $\eta$, and every positive unit of norm
one is a power of $\alpha$.
\end{proposition}

\begin{proof}
Suppose that a unit $v$ satisfies $1<v<\eta$.
Write $e=\operatorname{N}(v)\in\{1,-1\}$, so the conjugate of $v$
is $e/v$. The integral-basis description of $\mathcal O$ gives
\[
 v=\frac{x+yd}{2},\qquad
 x=v+e/v,\qquad y=\frac{v-e/v}{d},
\]
where $x,y$ are integers of the same parity. They are positive for
$v>1$, and $x^2-65y^2=4e$. Since $z+z^{-1}$ is strictly increasing
for $z>1$ and $\eta+\eta^{-1}=2d$, the inequality $1<v<\eta$
implies
\[
 0<y\le\frac{v+v^{-1}}d<2.
\]
Hence $y=1$, which would require $x^2=61$ or $x^2=69$; neither is
a square. Thus $\eta$ is the least positive unit greater than one.
For any positive unit $v$, choose $n\in\Z$ such that
$1\le v\eta^{-n}<\eta$. Minimality forces $v\eta^{-n}=1$.
Norms show that the norm-one units are precisely the even powers.
Finally, the elements $r+s\theta$ of norm $\pm1$ are exactly the
units. Multiplication by $-1$ makes their chosen real embedding
positive, after which the preceding argument applies. Uniqueness
follows from $\eta>1$ and the uniqueness of integral-basis coordinates.
\end{proof}

\begin{theorem}[Solutions under the linear constraints]\label{arith:constrained}
Let $a,b,c,N$ be positive integers and let $h\in\{1,-1\}$. Suppose
\begin{equation}\label{arith:constraints}
 a=2N-h,\qquad b+c=13N,\qquad
 a^3+b^3+c^3=561N^3-h.
\end{equation}
There is a unique $n\in\Z$ such that $h=(-1)^n$ and the tuple is
one of
\[
 (a_n,b_n,c_n,N_n),\qquad(a_n,c_n,b_n,N_n).
\]
The two alternatives are distinct.
\end{theorem}

\begin{proof}
Put $D=b-c$. Since
\[
 b^3+c^3=\frac{(b+c)^3+3(b+c)D^2}{4},
\]
substitution of the first two constraints into the third gives
\begin{equation}\label{arith:pell104}
 13D^2=5N^2+16hN-8,\qquad
 X^2-65D^2=104,\qquad X=5N+8h.
\end{equation}
If $h=1$, then $X>0$. If $h=-1$, the first equation rules out
$N=1$, and $N\ge2$ again gives $X>0$.

To solve $X^2-65D^2=104$ with $X>0$, reduce modulo $4$ to see that
$X,D$ have the same parity. They
cannot both be even, since division by $4$ would give
$x^2-y^2\equiv2\pmod4$. Thus both are odd. Reduction modulo
$13$ gives $13\mid X$; write $X=13x$, where $x$ is odd.
Choose $\varepsilon\in\{1,-1\}$ such that
$D\equiv\varepsilon x\pmod4$. Set
\[
 A=\frac{13x-5\varepsilon D}{4},\qquad
 B=\frac{D-\varepsilon x}{4}.
\]
Both are integers, and
\[
 A-\varepsilon B=\frac{7x-3\varepsilon D}{2}
\]
is even. Hence $A$ and $B$ have the same parity, and rationalizing
the denominator shows that
\begin{equation}\label{arith:symbolicdivision}
 \frac{X+Dd}{13+\varepsilon d}
 =\frac{A+Bd}{2}\in\mathcal O.
\end{equation}
This element has norm one. It is positive because
$X>|D|d$ and both $13+d$ and $13-d$ are positive.
Proposition~\ref{arith:unitcomplete} therefore gives
\begin{equation}\label{arith:104orbits}
 X+Dd=(13-d)\alpha^m
 \quad\text{or}\quad
 X+Dd=(13+d)\alpha^m,
 \qquad m\in\Z.
\end{equation}
Conversely, both displayed orbits have integral coordinates, norm
$104$, and positive first coordinate, because $\alpha$ and
$\alpha^{-1}=129-16d$ have integral coefficients and are positive
in both real embeddings.

Multiplication by $\alpha$ or by $\alpha^{-1}$ negates the first
coordinate modulo $5$. Both initial elements have first coordinate
$13\equiv3\pmod5$, so an element of either orbit at exponent $m$
has $X\equiv3(-1)^m\pmod5$. Since $X=5N+8h$, it follows that
$h=(-1)^m$.
The identity
\begin{align*}
 (13-d)(r_n+s_n\theta)^2
 &=13r_n^2+104r_ns_n+338s_n^2
       +d(-r_n^2+26s_n^2)\\
 &=(5N_n+8(-1)^n)+(b_n-c_n)d
\end{align*}
identifies the first orbit in \eqref{arith:104orbits} at exponent
$m$ with $(a_m,b_m,c_m,N_m)$. The second orbit at exponent $m$
is the conjugate of the first at exponent $-m$, and therefore gives
$(a_{-m},c_{-m},b_{-m},N_{-m})$. In both cases $N,a,b,c$ are
uniquely recovered from $X,D,h$ by
\[
 N=(X-8h)/5,\qquad a=2N-h,\qquad
 b=(13N+D)/2,\qquad c=(13N-D)/2.
\]

The two norm orbits are disjoint. An intersection would imply
$(13+d)/(13-d)=\alpha^k$ for an integer $k$, whereas
\[
 1<\frac{13+d}{13-d}<\alpha.
\]
Indeed, the ratio is $(117+13d)/52<129+16d=\alpha$. Within each orbit the exponent is
unique since $\alpha>1$. This proves uniqueness of the index and
alternative. Finally $D=0$ would force $X^2=104$, so $b\ne c$
and the two tuples at a fixed index are distinct.
\end{proof}

\subsection{Specialization at Pell parameters}
\begin{corollary}[Specialization at Pell parameters]\label{spec:pell}
Let
\begin{equation}\label{spec:parameters}
 u_n=s_n/r_n,\qquad k_n=a_n^3+b_n^3,
 \qquad n\in\Z,
\end{equation}
and put
\begin{equation}\label{spec:point}
 Q_n=\bigl(12G_n,\,36(a_n-b_n)G_n\bigr),
 \qquad G_n=a_n^2-a_nb_n+b_n^2.
\end{equation}
Every $Q_n\in E_{k_n}(\Q)$ has infinite order, and
\begin{equation}\label{spec:asymptotic}
 H_{E_{k_n}}(Q_n)=\frac43|n|\log\eta+O(1)
 \qquad(|n|\longrightarrow\infty).
\end{equation}
The positive integers $k_n$ occupy infinitely many classes in
$\Q^\times/\Q^{\times3}$.
\end{corollary}

\begin{proof}
The norm equation implies $\gcd(r_n,s_n)=1$, and the parity
following \eqref{arith:matrix} gives $r_n\ne0$. Thus every $u_n$
is defined. Homogeneity gives
\[
 k_n=r_n^6K(u_n),\qquad
 (x,y)\longmapsto(r_n^4x,r_n^6y):E_{K(u_n)}\longrightarrow E_{k_n},
\]
and this isomorphism carries $P(u_n)$ to $Q_n$. Infinite order
follows from Theorem~\ref{spec:global}; positivity of $k_n$ follows
from Proposition~\ref{arith:construction}.

Put
\[
 A=\frac{13-d}{2d}>0,\qquad B=\frac{13+d}{2d}>0,
 \qquad \eta'=8-d=-\eta^{-1}.
\]
Conjugating \eqref{arith:orbit} gives
\begin{equation}\label{spec:binet}
 r_n=-A\eta^n+B(\eta')^n,
 \qquad
 s_n=\frac{\eta^n-(\eta')^n}{d}.
\end{equation}
For $m\to\infty$, these formulas imply
\[
 \eta^{-m}(r_m,s_m)\longrightarrow(-A,1/d),
 \qquad
 (-1)^m\eta^{-m}(r_{-m},s_{-m})\longrightarrow(B,-1/d).
\]
Both limiting vectors are nonzero. Primitivity therefore gives
\[
 h(u_n)=\log\max\{|r_n|,|s_n|\}
       =|n|\log\eta+O(1).
\]
Combining this with Lemma~\ref{spec:comparison} and invariance
of canonical height under the displayed isomorphism proves
\eqref{spec:asymptotic}.

The parameters $u_n$ are distinct. In fact, on putting
$z_n=(\eta'/\eta)^n=(-\eta^{-2})^n$, formula
\eqref{spec:binet} becomes
\[
 u_n=\frac{1-z_n}{d(-A+Bz_n)}.
\]
The $z_n$ are distinct, and the fractional linear map is
nonconstant because $A\ne B$; its denominators do not vanish
at these indices since $r_n\ne0$. Theorem~\ref{spec:global} now shows that the $K(u_n)$ represent
infinitely many cube classes. Since $r_n^6$ is a rational cube,
the same is true of the $k_n$.
\end{proof}

\appendix
\section{Mordell--Weil groups and N\'eron--Severi lattices}\label{lat:section}
This appendix records the rank and integral lattice information
associated with the sections used in the main text. Root lattices
are positive definite, $L(a)$ rescales the pairing by $a$, and $U$
is the hyperbolic plane. All N\'eron--Severi and transcendental
lattices are geometric unless specified otherwise. For a complex
K3 surface $S$, write $T_S=\NS(S)^\perp\subset H^2(S,\Z)$.
Let $F_0$ denote the fibre class.

\subsection{The Mordell--Weil group}
For the degree-two map $\pi:C\to E_{\mathrm F}$ of
Lemma~\ref{geom:prym}, put
$A=\ker(\pi_*:J(C)\to E_{\mathrm F})^0$.
With origin $[1:-1:0]$, the identity $\pi_*\pi^*=[2]$ makes $\pi_*$
surjective, so $\dim A=3$. The homomorphism
$E_{\mathrm F}\times A\to J(C)$, $(e,a)\mapsto\pi^*e+a$,
has finite kernel: a kernel element has $2e=0$ and $a=-\pi^*e$.
Both sides have dimension four, proving the isogeny
$J(C)\sim E_0\times A$ over $\Q$.

Write
\[
\Hom_{\tau}(A,E_0)
=\{\phi\in\Hom_{\Qbar}(A,E_0):\phi\tau_* =\varrho\phi\}.
\]
\begin{theorem}[Mordell--Weil rank]
\label{geom:criterion}\label{geom:mw}
The subgroup
\[
\mathcal M=\Z P\oplus\Z R\oplus\langle T\rangle
\subset E_K(\Qbar(u))
\]
has torsion-free quotient, and
$\mathcal M\cap E_K(\Q(u))=\Z P$.
Writing $r=\rank E_K(\Qbar(u))$, one has
\begin{equation}\label{geom:rankformula}
 r=2+\rank_{\Z}\Hom_{\tau}(A,E_0),\qquad
 r\in\{2,4,6\},\qquad \rho(X)=14+r\in\{16,18,20\}.
\end{equation}
Moreover,
\[
1\le\rank E_K(\Q(u))\le r/2\le3.
\]
The following conditions are equivalent:
\begin{enumerate}
\item $\Hom_{\tau}(A,E_0)=0$;
\item $r=2$;
\item $E_K(\Qbar(u))=\mathcal M$.
\end{enumerate}
If these equivalent conditions hold, then
\[
E_K(\Q(\sqrt{-3})(u))=E_K(\Qbar(u))=\mathcal M,
\qquad E_K(\Q(u))=\Z P.
\]
\end{theorem}
\begin{proof}
The subgroup and its direct-sum description follow from the positive
height matrix of Proposition~\ref{lat:height} and the torsion
calculation in Lemma~\ref{geom:descent}.

Saturate $M=\Z[\omega]P$ in its rational span inside the free group.
The saturation is finitely generated: by~\eqref{geom:hom} it lies in
the finitely generated group $\Hom(J(C),E_0)$.
It is stable under $\varrho$, since both $M$ and the full free
section group are stable under $\varrho$.
It is a rank-one torsion-free module over the principal ideal
domain $\Z[\omega]$. Choose a section $Q$ representing a module
generator; then $P=aQ$ modulo torsion for some nonzero
$a\in\Z[\omega]$.
Applying the same $\varrho$-invariance calculation to $Q$ gives
$\langle aQ,aQ\rangle=N(a)\langle Q,Q\rangle$.
Thus $4/3=N(a)\langle Q,Q\rangle\ge N(a)(4/3)$, so $N(a)\le1$.
Since $a\ne0$, its norm is a positive integer. Therefore $N(a)=1$,
$a$ is a unit, and $M$ is already saturated.

Saturation of its free part and inclusion of all torsion show
that the quotient by $\mathcal M$ is torsion-free. Every element of $\mathcal M$ is defined
over $\Q(\sqrt{-3})$. Conjugation fixes $P$, sends $R$ to $-P-R$,
and sends $T$ to $-T$. Its invariant subgroup in $\mathcal M$ is
therefore exactly $\Z P$.

The degree-two map $\pi$ is equivariant: the automorphism induced
by $\tau$ on the Fermat quotient is $\varrho$ under
\eqref{geom:fermat}. Consequently $A$ is $\tau_*$-stable and the
isogeny $E_0\times A\to J(C)$ constructed above is
equivariant. After tensoring homomorphism groups with $\Q$, this
isogeny gives
\[
\Hom_{\tau}(J(C),E_0)\otimes\Q
\cong \End(E_0)\otimes\Q
\ \oplus\ \Hom_{\tau}(A,E_0)\otimes\Q.
\]
Here every endomorphism of $E_0$ commutes with $\varrho$, since
$\End_{\Qbar}(E_0)=\Z[\omega]$. Its contribution has rank two.
Lemma~\ref{geom:descent} proves the first identity
in~\eqref{geom:rankformula}. Homomorphism groups between abelian
varieties are torsion-free, so rank zero of the second summand
is equivalent to its vanishing.

The trivial lattice is $U\oplus A_2(-1)^{\oplus6}$, of rank
fourteen. Shioda--Tate therefore gives $\rho(X)=14+r$.
Since $X$ is a characteristic-zero K3 surface, $\rho(X)\le20$.
The action of $\varrho$ makes the free rational Mordell--Weil
space a vector space over $\Q(\omega)$, so $r$ is even.
The two independent sections give $r\ge2$, proving the stated
possibilities.

For the rational-rank bound, let $V$ be the free geometric group
tensored with $\Q$, and let $c$ denote complex conjugation.
Its action satisfies $c\varrho c=\varrho^2$. The invertible
operator $1+2\varrho$, whose square is $-3$, therefore exchanges
the $+1$ and $-1$ eigenspaces of $c$. Each eigenspace has
$\Q$-dimension $r/2$. Rational sections lie in the $+1$
eigenspace, giving the upper bound; the section $P$ gives the
lower bound.

The rank identity proves equivalence of the first two conditions.
If $r=2$, then the torsion-free quotient by $\mathcal M$ has
rank zero and hence vanishes. Conversely, equality with
$\mathcal M$ gives rank two. Under these conditions, the
field-of-definition and conjugation calculation above give the
final assertions.
\end{proof}

\subsection{Generators and discriminant form}\label{lat:geometric}
Denote the two fibres above $F=0$ by $f_1,f_2$, the two fibres
above $p/q=(1+\sqrt{-3})/2$ by $g_1,g_2$, and the two above
$p/q=(1-\sqrt{-3})/2$ by $g_3,g_4$. Each has two components not
meeting $O$. Label these components by $+$ and $-$ so that $P,R$
meet $g_i^+$ for every $i$, while $T$ meets
\[
 f_1^+,\quad f_2^+,\quad g_1^+,\quad g_2^+,\quad
 g_3^-,\quad g_4^-.
\]
The sections $P,R$ meet the identity components above $F=0$.

These assertions follow directly on the resolution. Near a simple
zero of $K$, use $s=K$ as a local base parameter. The chart
$x=sX_1$, $y=sY_1$ has equation
\[
 Y_1^2=sX_1^3-432.
\]
Its two exceptional components are $Y_1=\pm12\sqrt{-3}$.
At a root of $G$,
\[
 \frac{y_P}{K}=\frac{y_R}{K}
 =36\frac{p-q}{p+q}
 =\begin{cases}
 12\sqrt{-3},&p/q=(1+\sqrt{-3})/2,\\
 -12\sqrt{-3},&p/q=(1-\sqrt{-3})/2.
 \end{cases}
\]
The corresponding $X_1$ coordinates of $P,R,T$ are respectively
$12/(p+q)$, $12\omega/(p+q)$, and $0$, which are distinct.
Thus the sections are pairwise disjoint above these fibres.
At a root of $F$, the $x$-coordinates of $P,R$ are nonzero, while $T$
meets an exceptional component. Away from $K=0$, equality of the
$x$-coordinates of $P$ and $R$ would force $G=0$, and equality with the
$x$-coordinate of $T$ would have the same consequence. At infinity the
minimal coordinates of $P,R$ have distinct nonzero $x$-coordinates,
and that of $T$ is zero. Their affine coordinates also show that
all three sections are disjoint from $O$.

Let $L$ be the sublattice of $\NS(X_{\Qbar})$ generated by the
trivial lattice and the sections $P,R,T$. It is generated by the
following seventeen classes:
\begin{equation}\label{lat:generators}
 O,F_0,f_1^+,f_1^-,f_2^+,f_2^-,
 g_1^+,g_1^-,\ldots,g_4^+,g_4^-,P,R,T.
\end{equation}
Every section and every listed component has self-intersection
$-2$. The fibre has square zero, meets each section once, and is
orthogonal to every component. The two nonidentity components in
the same fibre intersect once. Section--component intersections
are one exactly for the incidences specified above. All other
intersections between distinct listed classes vanish.

The Shioda image of a torsion section is zero. Applying its local
correction terms to $T$ gives
\begin{align}\label{lat:torsionrelation}
 3O+6F_0-3T={}&2f_1^++f_1^-+2f_2^++f_2^-\notag\\
 &+2g_1^++g_1^-+2g_2^++g_2^-\notag\\
 &+g_3^++2g_3^-+g_4^++2g_4^-.
\end{align}
The relation initially holds in $\NS(X_{\Qbar})\otimes\Q$ and
holds integrally after clearing denominators because the
N\'eron--Severi group of a K3 surface is torsion-free. Its
$f_1^-$ coefficient is one. Removing that generator from
\eqref{lat:generators} therefore leaves sixteen generators.
The Shioda--Tate quotient identifies
$L/\Triv(X)$ with $\mathcal M$ and
$\NS(X_{\Qbar})/L$ with $E_K(\Qbar(u))/\mathcal M$
\cite{Shioda,SchuttShioda}. The latter is torsion-free by
Theorem~\ref{geom:criterion}; hence $L$ is primitive. Its rank is
$14+2=16$, so the sixteen remaining generators constitute an
integral basis, denoted by $\mathcal B$.

For a direct determinant calculation, put
$L_0=\Triv(X)+\Z P+\Z R$. The trivial lattice is
$U\oplus A_2(-1)^{\oplus6}$, of determinant $-3^6$.
Orthogonal projection of $P,R$ away from the trivial lattice has
Gram matrix equal to the negative of their height matrix.
The Schur-complement formula for this block Gram matrix therefore
gives
\[
 \det L_0=(-3^6)\det\!\left(
 -\frac23\begin{pmatrix}2&-1\\-1&2\end{pmatrix}\right)=-972.
\]
The Shioda--Tate quotient identifies $L/L_0$ with
$\mathcal M/(\Z P+\Z R)=\langle T\rangle$, of order three.
Thus $[L:L_0]=3$, and
\begin{equation}\label{lat:disc}
 \det L
 =\frac{\det L_0}{[L:L_0]^2}=\frac{-972}{3^2}=-108.
\end{equation}
The discriminant group and form of $L$ are
\[
 A_L=L^*/L,\qquad q_L(\bar v)=v^2\pmod{2\Z},\qquad
 b_L(\bar v,\bar w)=v\cdot w\pmod\Z.
\]
Consider the dual vectors
\begin{align*}
 a&=(P+R)/2,&
 b&=(O+P+g_1^-+g_2^-+g_3^-+g_4^-)/2,\\
 u&=(f_2^++2f_2^-+g_2^++2g_2^-+2g_4^++g_4^-)/3,\\
 v&=(f_2^++2f_2^-+g_1^++2g_1^-+2g_3^++g_3^-)/3,\\
 w&=(2g_1^++g_1^-+g_2^++2g_2^-+
       2g_3^++g_3^-+g_4^++2g_4^-)/3.
\end{align*}
Each pairs integrally with the generators in
\eqref{lat:generators}, and hence belongs to $L^*$. To verify
the pairings for $u,v,w$, a component pair with coefficients
$(1,2)/3$ has pairings $(0,-1)$ with the two components, and a
pair with coefficients $(2,1)/3$ has pairings $(-1,0)$.
Each of $u,v$ pairs to $1$ with each of $P,R,T$, while $w$
pairs to $2$ with each of them. All three pair trivially with $O,F_0$.
For $a,b$, the component incidences above show directly that
all pairings are integral.
The Gram matrices on these two groups of vectors are
\begin{equation}\label{lat:primarygrams}
 \begin{pmatrix}-1&-1/2\\-1/2&-3\end{pmatrix},
 \qquad
 \begin{pmatrix}
 -2&-2/3&-1\\
 -2/3&-2&-1\\
 -1&-1&-8/3
 \end{pmatrix}.
\end{equation}
The classes of $a,b$ have order dividing two, and their induced
bilinear form is nondegenerate, so they generate a subgroup
isomorphic to $(\Z/2)^2$. The same argument with $u,v,w$ gives
$(\Z/3)^3$. These primary groups are orthogonal: a pairing
between them is annihilated by both two and three. Their orders
multiply to $108=|\det L|$, so they exhaust $A_L$.

Write $v(2)$ for the quadratic form on $(\Z/2)^2$ taking value
$1$ modulo $2\Z$ on every nonzero element, and $u(3)$ for the
form on $(\Z/3)^2$ with an isotropic basis of mutual pairing
$1/3$ modulo $\Z$. The first matrix in
\eqref{lat:primarygrams} gives $v(2)$. The classes of $u,v$
give $u(3)$, while $w$ is orthogonal to them modulo $\Z$ and
has square $-8/3\equiv-2/3\pmod{2\Z}$. Therefore
\begin{equation}\label{lat:discriminantform}
 A_L\cong(\Z/6)^2\oplus\Z/3,
 \qquad q_L\cong v(2)\oplus u(3)\oplus\langle-2/3\rangle.
\end{equation}

\begin{theorem}[The lattice $L$ and its orthogonal complement]\label{lat:full}
After fixing an embedding into $\C$, let
$W=L^\perp\subset H^2(X(\C),\Z)$. Then
\begin{align}\label{lat:fullidentification}
 L&\cong U\oplus E_6(-1)\oplus D_4(-1)
                 \oplus A_2(-1)^{\oplus2},\notag\\
 W&\cong U\oplus U(3)\oplus A_2(-2).
\end{align}
There are primitive inclusions $L\subset\NS(X_{\Qbar})$ and
$T_X\subset W$. They are equalities precisely when
$\rank E_K(\Qbar(u))=2$.
\end{theorem}

\begin{proof}
The rank-sixteen lattice in~\eqref{lat:fullidentification} has
signature $(1,15)$. The
discriminant forms of $E_6(-1)$, $D_4(-1)$, and $A_2(-1)$ are,
respectively, $\langle2/3\rangle$, $v(2)$, and
$\langle-2/3\rangle$. Moreover
\[
 \langle2/3\rangle\oplus\langle-2/3\rangle\cong u(3):
\]
if $e,f$ are the standard generators, then $e+f,e-f$ are an
isotropic basis and their mutual pairing is $4/3\equiv1/3$
modulo $\Z$. Hence this lattice has discriminant form
\eqref{lat:discriminantform}. The uniqueness theorem for
indefinite even lattices applies because the rank is sixteen
and the discriminant group has length three
\cite[Theorem~1.14.2]{Nikulin}.

The lattice $L$ is primitive in $\NS(X_{\Qbar})$, which is
primitive in the even unimodular K3 cohomology lattice. Thus $L$
is primitive there as well, and its orthogonal complement $W$ has signature
$(2,4)$ and discriminant form $-q_L$. For the lattice
$A_2(-2)$ with basis $e,f$ and Gram matrix
\[
 \begin{pmatrix}-4&2\\2&-4\end{pmatrix},
\]
the classes of $e/2,f/2$ generate its $2$-primary discriminant
form $v(2)$. The class $(e-f)/3$ generates its $3$-primary part
and has square $-4/3\equiv2/3\pmod{2\Z}$.
The discriminant form of $U(3)$ is $u(3)$, which is isomorphic
to its negative by changing the sign of one basis vector.
Similarly $v(2)$ is isomorphic to its negative. Thus
$U\oplus U(3)\oplus A_2(-2)$ has the required signature and
discriminant form $-q_L$. Its rank is six and its discriminant
group has length three, so the same uniqueness theorem applies.
Since $L\subset\NS(X_{\Qbar})$, its orthogonal complement contains
$T_X$ primitively. By Shioda--Tate, $\NS(X_{\Qbar})$ has rank
sixteen exactly when the geometric Mordell--Weil rank is two.
Primitivity then turns equality of ranks into $L=\NS(X_{\Qbar})$,
equivalently $W=T_X$.
\end{proof}

\subsection{Galois invariants}\label{lat:arithmetic}

\begin{theorem}[Galois invariants]\label{lat:rational}
The splitting field of the singular fibres has Galois group
$D_4\times\Z/2$ over $\Q$, where $D_4$ denotes the dihedral
group of order eight. One has
\begin{equation}\label{lat:rationalidentification}
 L^{G_{\Q}}\cong U\oplus\langle-4\rangle^{\oplus4}.
\end{equation}
It is primitive in $\NS(X_{\Qbar})^{G_{\Q}}$, and each of its
classes has a divisor representative defined over $\Q$.
Writing $r_{\Q}=\rank E_K(\Q(u))$, one has
\[
\rank\NS(X_{\Qbar})^{G_{\Q}}=5+r_{\Q},
\qquad 1\le r_{\Q}\le3.
\]
\end{theorem}

\begin{proof}
Put $k=\Q(\sqrt{-3})$ and $\delta=4433+3016\omega$.
The quadratic factors $p+\omega q$ and $p+\omega^2q$ of $G$
have discriminants $\delta$ and $\bar\delta$. Neither is a
square in $k$, since
\[
 \delta\bar\delta=91^2\cdot1857
\]
is not a square in $\Q$, whereas the norm of a square in $k$
is a square in $\Q$. Their product is not a square in $k$
either. Indeed, a rational number is a square in $k$ only if
it or its quotient by $-3$ is a rational square; neither
condition holds for $91^2\cdot1857$.

Every root $u_0$ of $G$ generates a field containing $k$, since
$q(u_0)\ne0$ and $-p(u_0)/q(u_0)$ is one of $\omega,\omega^2$.
Consequently the splitting field of $G$ is
$k(\sqrt\delta,\sqrt{\bar\delta})$, of degree four over $k$.
Choose the square roots as complex conjugates. Complex
conjugation has order two and interchanges the two independent
quadratic generators. The Galois group over $\Q$ is therefore
the semidirect product of $(\Z/2)^2$ by the transposition of
its generators, namely $D_4$.
Its three quadratic subfields are
\[
 \Q(\sqrt{-3}),\qquad
 \Q(\sqrt{1857}),\qquad
 \Q(\sqrt{-619}).
\]
The discriminant of $F$ is $-5655$, so its root field
$\Q(\sqrt{-5655})$ is different from these three fields, as
their defining rational numbers have distinct square classes.
It is therefore linearly disjoint from the splitting field of
$G$, giving the claimed Galois group. The component splitting
fields require only $\sqrt{-3}$ in addition to the fibre root,
and this element is already in the displayed splitting field.

Above the roots of $F$, the two nonidentity components are
interchanged by adjoining $\sqrt{-3}$. Since the quadratic
fields $\Q(\sqrt{-3})$ and $\Q(\sqrt{-5655})$ are distinct,
the four components above $F=0$ form one Galois orbit.
Above a root $u_0$ of $G$, the field $\Q(u_0)$ already contains
$\sqrt{-3}$, so both nonidentity components are defined over
that root field. The rational section $P$ distinguishes the
component labelled $+$ from that labelled $-$, giving two
Galois orbits, each consisting of four components.

Set
\[
 S_f=\sum_{i=1}^2(f_i^++f_i^-),\qquad
 S_+=\sum_{i=1}^4g_i^+,\qquad
 S_-=\sum_{i=1}^4g_i^-.
\]
The invariant part of the trivial lattice has rational rank
$2+1+2=5$. By Theorem~\ref{geom:criterion}, the quotient
$L/\Triv(X)\cong\mathcal M$ has invariant rank one, generated by $P$.
The Galois action on $L$ factors through a finite group, so taking
invariants is exact after tensoring with $\Q$. It follows that
\[
 O,F_0,S_f,S_+,S_-,P
\]
form a basis of $(L\otimes\Q)^{G_{\Q}}$.

They are also an integral basis of $L^{G_{\Q}}$.
To prove this, express an invariant class in $L$ as a rational
linear combination of these six classes and expand in the
geometric basis $\mathcal B$. Equation
\eqref{lat:torsionrelation} gives
\begin{align*}
 S_f={}&3O+6F_0-3T-f_1^+-f_2^+\\
 &-2g_1^+-g_1^--2g_2^+-g_2^-
   -g_3^+-2g_3^--g_4^+-2g_4^-.
\end{align*}
The $f_1^+$ coordinate forces the coefficient of $S_f$ to be
integral. The $g_1^+$ and $g_1^-$ coordinates then force the
coefficients of $S_+$ and $S_-$ to be integral. Finally the
$P,O,F_0$ coordinates force all the remaining coefficients to
be integral. Each generator is represented by a divisor
defined over $\Q$: the component sums are Galois-orbit sums,
the sections $O,P$ are rational, and the fibre class has a
rational representative. Hence every class of $L^{G_{\Q}}$ has such a representative.

Their Gram matrix, in the stated order, is
\begin{equation}\label{lat:rationalgram}
 \begin{pmatrix}
 -2&1&0&0&0&0\\
 1&0&0&0&0&1\\
 0&0&-4&0&0&0\\
 0&0&0&-8&4&4\\
 0&0&0&4&-8&0\\
 0&1&0&4&0&-2
 \end{pmatrix}.
\end{equation}
The classes $O+F_0,F_0$ span a unimodular copy of $U$.
Put $S=P-O-2F_0$. An integral basis of its orthogonal
complement is
\[
 S_f,\qquad S,\qquad S_++S,\qquad S_-+S_++S.
\]
The change of basis is unimodular, and
\eqref{lat:rationalgram} shows that these four vectors are
pairwise orthogonal, each with square $-4$. This proves
\eqref{lat:rationalidentification}.
Primitivity of $L$ implies primitivity of its invariant subgroup
in $\NS(X_{\Qbar})^{G_{\Q}}$: if a nonzero integer multiple
of an invariant class lies in $L$, then so does the class itself.

Finally, the Galois action on the full N\'eron--Severi group factors
through a finite quotient: choose divisors representing a finite
basis and a common number field over which they are defined.
Tensor the full Shioda--Tate quotient
$\NS(X_{\Qbar})/\Triv(X)\cong E_K(\Qbar(u))$ with $\Q$
and take Galois invariants. The invariant trivial lattice has rank
five by the orbit calculation above. The invariant rational
Mordell--Weil space has dimension $r_{\Q}$: a common denominator
and the exponent of the finite torsion group turn each invariant
rational vector into a rational section. Exactness of invariants
over $\Q$ for this finite quotient gives the claimed identity. The bounds
on $r_{\Q}$ are Theorem~\ref{geom:criterion}.
\end{proof}

\section{The symplectic quotient and its lattices}\label{quot:section}
We record the symplectic quotient of the quadratic base change in
Proposition~\ref{geom:basechange}, its lattice correspondence,
and the analogous rank identities for other smooth branch fibres.

Put
\[
D(t)=65+1352t-4368t^2,
\]
and let $Z$ be the minimal smooth elliptic surface associated with
\begin{equation}\label{quot:equation}
E^D:\qquad y^2=x^3-432(t^3+1)^2D(t)^3.
\end{equation}
All lattices and groups of sections in this section are geometric.
The construction from a quadratic base change and fibrewise inversion
is standard~\cite[Section~5B]{HulekSchutt}.

\begin{theorem}[The symplectic quotient]\label{quot:main}
Let $\iota$ be the covering involution of $X\to Y$, and let $[-1]$
denote fibrewise inversion. The involution $\sigma=[-1]\circ\iota$
is symplectic and has eight fixed points. The minimal resolution of
$X/\langle\sigma\rangle$ is $Z$.
Its singular fibres are three fibres of type $\mathrm{IV}$ and two
fibres of type $I_0^*$. Its group $E^D(\Qbar(t))$ is torsion-free.
The trivial lattice
\begin{equation}\label{quot:trivial}
N_Z=U\oplus A_2(-1)^{\oplus3}\oplus D_4(-1)^{\oplus2}
\end{equation}
is a primitive sublattice of $\NS(Z)$, and
\begin{equation}\label{quot:rankidentity}
\rank\MW(X)=2+\rank\MW(Z),\qquad \rho(X)=\rho(Z).
\end{equation}
The rank of $\MW(Z)$ belongs to $\{0,2,4\}$.
The quotient correspondence induces finite-index embeddings
\begin{equation}\label{quot:inclusions}
T_Z(2)\hookrightarrow T_X,
\qquad T_X(2)\hookrightarrow T_Z,
\end{equation}
whose cokernels are killed by $2$.
\end{theorem}

\begin{proof}
The quadratic extension associated with $t=p(u)/q(u)$ is
\[
\Q(t)\subset\Q(t,v),\qquad v^2=D(t).
\]
Indeed, $p(u)-tq(u)=(26-182t)u^2-13u+(1-6t)$ has discriminant
$D(t)$. In the coordinates
\[
Y:\quad Y_1^2=X_1^3-432(t^3+1)^2,
\]
the involution is
\[
\sigma(v,X_1,Y_1)=(-v,X_1,-Y_1).
\]
The invariant functions $x=DX_1$ and $y=DvY_1$ satisfy
\eqref{quot:equation}. They generate the invariant function field:
adjoining $v$ recovers $X_1=x/D$ and $Y_1=y/(Dv)$, and $v$ is
negated by $\sigma$.

Both the deck involution and fibrewise inversion negate the holomorphic
two-form, so their composition preserves it. A fixed point must lie
on a ramification fibre. These are two smooth fibres, and on each
$\sigma$ is inversion, with precisely four points killed by $2$.
Thus there are eight fixed points. A base parameter and a fibre
parameter are both negated at each such point, so the quotient has
eight ordinary double points and their minimal resolutions are crepant.

The roots of $D$ are distinct and avoid the roots of $t^3+1$.
At these two kinds of zeros the valuations of $a_6$ in
\eqref{quot:equation} are respectively $3$ and $2$, giving minimal
fibres of types $I_0^*$ and $\mathrm{IV}$ in characteristic zero
\cite{Miranda,SchuttShiodaSurvey}. At infinity $a_6$ has degree
twelve and nonzero leading coefficient. The chart
$s=1/t$, $\widetilde x=s^4x$, $\widetilde y=s^6y$ therefore has
a smooth fibre. The fundamental line bundle is $\mathcal O(2)$;
the canonical bundle formula and the standard calculation of
$R^1f_*\mathcal O_Z$ give $K_Z\cong\mathcal O_Z$ and
$H^1(Z,\mathcal O_Z)=0$. Hence $Z$ is K3.
The invariant two-form on the quotient extends without zeros over
the crepant resolution, and the resolution has vanishing
$H^1(\mathcal O)$, since taking invariants is exact in characteristic
zero and ordinary double points are rational singularities.
It is K3 as well. Birational smooth K3 surfaces are isomorphic,
proving the quotient assertion.

Let $k=\Qbar$, $F=k(t)$, $L=F(v)$, and let $s$ generate
$\operatorname{Gal}(L/F)$. Write $E/F$ for the generic fibre of $Y$.
The isomorphism
\[
E^D_L\longrightarrow E_L,\qquad (x,y)\longmapsto(x/D,y/(Dv))
\]
identifies $E^D(F)$ with $\{Q\in E(L):sQ=-Q\}$.
The invariant subgroup is $E(F)$. Since
\[
2Q=(Q+sQ)+(Q-sQ)
\]
and the intersection of the two subgroups is killed by $2$,
tensoring with $\Q$ gives
\begin{equation}\label{quot:rankdecomposition}
\rank E(L)=\rank E(F)+\rank E^D(F).
\end{equation}
The rational elliptic surface $Y$ has Picard rank ten and three
type-IV fibres, so Shioda--Tate gives $\rank E(F)=10-2-6=2$.
This proves the first identity in~\eqref{quot:rankidentity}.

By Lemma~\ref{geom:descent}, the torsion of $E(L)$ is the
order-three subgroup. It comes from the sections
$(0,\pm12\sqrt{-3}(t^3+1))$ on $Y$, so it is invariant under $s$.
Torsion pulled back from $E^D(F)$ is also anti-invariant, hence
killed by $2$. It must therefore be zero.

The quotient of $\NS(Z)$ by its trivial lattice is its group of
sections~\cite{Shioda,SchuttShioda}. The latter is torsion-free,
so the trivial lattice is primitive. The fibre types give exactly
\eqref{quot:trivial}. Shioda--Tate now yields
\[
\rho(X)=14+\rank\MW(X),\qquad
\rho(Z)=16+\rank\MW(Z),
\]
proving equality of the Picard numbers. The action of $\Z[\omega]$
makes each free Mordell--Weil rank even, and $\rho(Z)\le20$ gives
the three possibilities stated for $\MW(Z)$.

Finally, blow up the eight fixed points, obtaining
$b:\widetilde X\to X$ and a degree-two morphism
$f:\widetilde X\to Z$. To describe the quotient by the lifted
involution, use equivariant local coordinates $(a,b)$ at a fixed
point, with involution $(a,b)\mapsto(-a,-b)$. On a blowup chart
$b=az$, the lifted involution is $(a,z)\mapsto(-a,z)$; its quotient
has smooth coordinates $(a^2,z)$. The other chart is identical.
If $E_i$ is an exceptional curve and $B_i$ its image, then
$f^*B_i=2E_i$, $f_*E_i=B_i$, and the projection formula gives
$B_i^2=2E_i^2=-2$. Thus the quotient of the blowup is smooth
and is the minimal resolution of the eight ordinary double points.
The map $f$ is finite of degree two.

Blowups add only algebraic classes, so
$T_{\widetilde X}=b^*T_X$ integrally; we identify these lattices.
The integral pullback and pushforward satisfy
\[
 f_*f^*=2\operatorname{id},\qquad
 f^*f_*=\operatorname{id}+\widetilde\sigma^*.
\]
The second identity is the correspondence identity obtained from
the diagonal and the graph of the lifted involution in
$\widetilde X\times_Z\widetilde X$. At the branch locus its local
equation is $(a_1-a_2)(a_1+a_2)=0$, so both components have
multiplicity one and there is no additional branch component.
The projection formula
shows that $f^*T_Z\subset T_X$ and $f_*T_X\subset T_Z$, since
pullback and pushforward of divisor classes are algebraic.
It also gives $(f^*c,f^*d)=2(c,d)$.
The involution acts identically on $T_X$: its anti-invariant
rational Hodge subspace has only type $(1,1)$ because $\sigma$
is symplectic, and is therefore algebraic by Lefschetz's theorem.
It has zero intersection with $T_X\otimes\Q$, since the
N\'eron--Severi pairing is nondegenerate by the Hodge index theorem.
Consequently $f^*f_*a=2a$ for $a\in T_X$, while $f_*f^*c=2c$
for $c\in T_Z$. For $a,b\in T_X$, the projection formula now gives
$(f_*a,f_*b)=(a,f^*f_*b)=2(a,b)$. Both maps are injective;
the displayed composition identities show that the quotients of
their images are killed by two. Equality of Picard numbers gives
equality of transcendental ranks, so both indices are finite.
\end{proof}

\begin{corollary}[Orthogonal complement of the trivial lattice]
\label{quot:conditional}
The orthogonal complement of the primitive trivial lattice is
\[
W_Z=N_Z^\perp\subset H^2(Z(\C),\Z),\qquad
W_Z\cong U(2)\oplus U(6)\oplus A_2(-1).
\]
It contains $T_Z$ primitively.
The following conditions are equivalent:
\[
\rank\MW(X)=2,\qquad \MW(Z)=0,\qquad \NS(Z)=N_Z.
\]
If they hold, then
\begin{equation}\label{quot:lattices}
T_Z\cong U(2)\oplus U(6)\oplus A_2(-1),
\end{equation}
and the two embeddings in~\eqref{quot:inclusions} have indices
$16$ and $4$, respectively.
\end{corollary}

\begin{proof}
The rank identity, torsion-freeness and exact quotient
$\NS(Z)/N_Z\cong\MW(Z)$ prove the equivalences.
The primitive lattice $N_Z$ has signature $(1,15)$ and
discriminant $-432$. The opposite of its discriminant form has
primary parts
\begin{equation}\label{quot:targetforms}
v(2)^{\oplus2},\qquad\langle2/3\rangle^{\oplus3}.
\end{equation}
Here $v(2)$ is the form on $(\Z/2)^2$ taking value one on every
nonzero element. The discriminant forms of $D_4(-1)$ and
$A_2(-1)$ are $v(2)$ and $\langle-2/3\rangle$.

Put $L_0=U(2)\oplus U(6)\oplus A_2(-1)$. Its signature is
$(2,4)$ and its $2$-primary form is $u(2)^{\oplus2}$, where
$u(2)$ has an isotropic basis with mutual pairing $1/2$.
This equals $v(2)^{\oplus2}$: for orthogonal hyperbolic pairs
$a,b$ and $c,d$, the pairs
\[
(a+b,\ a+c+d),\qquad(a+b+d,\ c+d)
\]
are orthogonal anisotropic planes. For its $3$-primary form,
choose $x,y,z$ with
\[
q(x)=q(y)=0,\quad b(x,y)=2/3,\quad q(z)=-2/3,
\quad b(x,z)=b(y,z)=0.
\]
Then $x-y,x+y+z,x+y-z$ form an orthogonal basis with square
$2/3$ modulo $2\Z$. This proves that $L_0$ has the form
\eqref{quot:targetforms}. The primitive embedding of $N_Z$
into the unimodular K3 lattice gives $W_Z$ the opposite form
and signature $(2,4)$. Its discriminant group has length four,
and rank six is at least that length plus two. The uniqueness
theorem for indefinite even lattices
\cite[Theorem~1.14.2]{Nikulin} therefore gives the asserted
identification of $W_Z$. The inclusion $N_Z\subset\NS(Z)$
gives a primitive inclusion $T_Z\subset W_Z$, and under the
equivalent rank conditions it is an equality, proving
\eqref{quot:lattices}.

Under the same rank condition, Theorem~\ref{lat:full} gives
$\NS(X)=L$ and $T_X=L^\perp$. Hence $\det T_X=108$, while
$\NS(Z)=N_Z$ gives $\det T_Z=432$. Both have rank six.
The squares of the two indices are
\[
\frac{432\,2^6}{108}=16^2,
\qquad
\frac{108\,2^6}{432}=4^2,
\]
as claimed.
\end{proof}

\subsection{Variation of the branch values}\label{fam:section}
\begin{corollary}[Quadratic base change]\label{fam:rank}
Let $k$ be algebraically closed of characteristic zero, and let
$\lambda,\mu\in k$ satisfy
\[
(\lambda-\mu)(\lambda^3+1)(\mu^3+1)\ne0.
\]
Put
\[
p_{\lambda,\mu}=\lambda u^2-\mu,\qquad
q_{\lambda,\mu}=u^2-1,\qquad
K_{\lambda,\mu}=p_{\lambda,\mu}^3+q_{\lambda,\mu}^3.
\]
Let $X_{\lambda,\mu}$ be the minimal elliptic surface of
$y^2=x^3-432K_{\lambda,\mu}^2$ and $Z_{\lambda,\mu}$ that of
\begin{equation}\label{fam:twist}
y^2=x^3-432(t^3+1)^2(t-\lambda)^3(t-\mu)^3.
\end{equation}
Both surfaces are K3, with singular fibres respectively
$6\,\mathrm{IV}$ and $3\,\mathrm{IV}+2\,I_0^*$, and
\begin{align}\label{fam:identities}
\rank\MW(X_{\lambda,\mu})
&=2+\rank\MW(Z_{\lambda,\mu}),\\
\rho(X_{\lambda,\mu})&=\rho(Z_{\lambda,\mu}).\nonumber
\end{align}
Consequently
\[
\rank\MW(X_{\lambda,\mu})\in\{2,4,6\},\qquad
\rho(X_{\lambda,\mu})\in\{16,18,20\}.
\]
The rank of $\MW(X_{\lambda,\mu})$ exceeds two if and only if
\eqref{fam:twist} has a section of infinite order.
\end{corollary}

\begin{proof}
The morphism $t=(\lambda u^2-\mu)/(u^2-1)$ has degree two and
branch values $\lambda,\mu$, which are distinct and avoid the
singular fibres of $Y$. The sextic $K_{\lambda,\mu}$ is therefore
squarefree, with nonzero leading coefficient $\lambda^3+1$.
The base-change argument of Proposition~\ref{geom:basechange}
gives the K3 surface $X_{\lambda,\mu}$ with six type-IV fibres.
Since $u^2=(t-\mu)/(t-\lambda)$, the function
$v=(t-\lambda)u$ satisfies $v^2=(t-\lambda)(t-\mu)$.
The quotient construction and fibre calculation in
Theorem~\ref{quot:main}, with $D=(t-\lambda)(t-\mu)$, give
\eqref{fam:twist}, the stated fibres, and the K3 property of
$Z_{\lambda,\mu}$.

The decomposition~\eqref{quot:rankdecomposition} gives the first
identity in~\eqref{fam:identities}, since $\rank\MW(Y)=2$.
The trivial lattices have ranks fourteen and sixteen, so
Shioda--Tate gives the second identity. To apply the characteristic-zero
bound $\rho\le20$, choose an algebraic closure $k_0\subset k$ of
$\Q(\lambda,\mu)$ and an embedding $k_0\hookrightarrow\C$.
The surfaces are defined over $k_0$, and geometric N\'eron--Severi
groups are unchanged by algebraically closed field extension.
Their Picard ranks therefore agree with those of the corresponding
complex K3 surfaces. The action of $\Z[\omega]$ makes the section
ranks even, giving the stated possibilities. The final
equivalence follows from the first rank identity.
\end{proof}

\end{document}